\documentclass[11pt,a4paper]{amsart}
\usepackage[english]{babel}
\usepackage[T1]{fontenc}
\usepackage[utf8]{inputenc}

\usepackage{amsthm}
\usepackage{amsmath}
\usepackage{amssymb}
\usepackage{mathrsfs}
\usepackage{mathtools}
\usepackage{stmaryrd}
\usepackage{csquotes}
\usepackage{mathptmx}
\usepackage{shuffle}
\usepackage{bm} 
\usepackage{comment}
\usepackage{physics}

\usepackage[section]{placeins}

\usepackage{graphicx}
\usepackage{hyperref}
\usepackage{lmodern}
\usepackage{fullpage}
\usepackage{fancybox}
\usepackage{xcolor}
\usepackage{relsize}
\usepackage{ifthen}
\usepackage{tikz}
\usepackage{tikz-cd}
\usetikzlibrary{matrix}

\usepackage[left=2cm, right=2cm, top=3cm, bottom=3cm]{geometry}

\newcommand\bb{\mathbb}

\newcommand\End{\text{End}}

\DeclareMathOperator{\GL}{GL}
\DeclareMathOperator{\Id}{Id}

\newcommand\dif{\text{dif}}
\newcommand\inv{\text{inv}}
\newcommand\NCP{\text{NCP}}

\newcommand\boxconv{{~\boxed{\!*\!}|~}}
\newcommand\boxconvred{{~\overline{\boxed{\!*\!}}~}}
\newcommand\boxconvredred{{~|\boxed{\!*\!}|~}}
\newcommand\boxconvline{{~{|\boxed{\!*\!}}~}}

\newcommand\Mult{\text{Mult}}
\newcommand\PT{\text{PT}}
\newcommand\ST{\text{ST}}
\newcommand\RST{\text{RST}}
\newcommand\LST{\text{LST}}
\newcommand\rot{\text{rot}}
\newcommand\Edel{\text{Edel}}
\newcommand\Prod{\text{Prod}}
\newcommand\gaps{\text{gaps}}
\newcommand\NDPF{\text{NDPF}}

\newcommand\onet{\begin{tikzpicture} \draw[fill](0,0) circle (.05); \end{tikzpicture}} 

\newcommand\children[3]{
\begin{tikzpicture}[scale = .8, baseline={(0,0)}]
\draw[fill](0,0) circle (.05);
\draw (-.3,.3) -- (0,0) -- (.3,.3) (0,0)--(0,-.3);
\draw (-.4,.4) node{$#1$};
\draw (.4,.4) node{$#3$};
\draw (0.18,-0.2) node{$#2$};
\end{tikzpicture}
}

\newcommand\tunder[1]{{\begin{tikzpicture}[scale = .25, baseline={(0,0)}] \draw (0,-.5)--(0,0)--(.5,.5) (0,0)--(-.5,.5); \draw (0,.8) node[right]{$#1$}; \end{tikzpicture}}}
\newcommand\tover[1]{{\begin{tikzpicture}[scale = .25, baseline={(0,0)}] \draw (0,-.5)--(0,0)--(.5,.5) (0,0)--(-.5,.5); \draw (0,.8) node[left]{$#1$}; \end{tikzpicture}}}
\newcommand\tovunder[2]{{\begin{tikzpicture}[scale = .3, baseline={(0,0)}] \draw (0,-.5)--(0,0)--(.5,.5) (0,0)--(-.5,.5); \draw (0,0) node[above left]{$#1$}; \draw (0,0) node[above right]{$#2$}; \end{tikzpicture}}}
\newcommand\tovovunder[3]{{\begin{tikzpicture}[scale = .3, baseline={(0,0.1)}] \draw (0,-.5)--(0,0)--(.5,.5) (0,0)--(-1.5,1.5) (-1,1)--(-.5,1.5); \draw (-1,1.9) node[left]{$#1$}; \draw (-1,1.9) node[right]{$#2$}; \draw (0,.7) node[right]{$#3$}; \end{tikzpicture}}}
\newcommand\trecb[2]{\tovovunder{}{#1}{#2}}
\newcommand\toverrightcomb[3]{{\begin{tikzpicture}[scale = .3, baseline={(0,0.4)}] \draw (0,-.5)--(0,0)--(3.5,3.5) (0,0)--(-.5,.5) (1,1)--(.5,1.5) (3,3)--(2.5,3.5); \draw (1.1,2.7) node{\rotatebox[origin=c]{80}{$\ddots$}} (0,0) node[above left]{$#1$} (1,1) node[above left]{$#2$} (3,3) node[above left]{$#3$}; 
\end{tikzpicture}}}

\newcommand\smtunder[1]{{\begin{tikzpicture}[scale = .15, baseline={(0,0)}] \draw (0,-.5)--(0,0)--(.5,.5) (0,0)--(-.5,.5); \draw (-.3,.8) node[right]{$\scriptstyle#1$}; \end{tikzpicture}}}
\newcommand\smtover[1]{\!{\begin{tikzpicture}[scale = .15, baseline={(0,0)}] \draw (0,-.5)--(0,0)--(.5,.5) (0,0)--(-.5,.5); \draw (.4,.8) node[left]{$\scriptstyle#1$}; \end{tikzpicture}}}
\newcommand\smtovunder[2]{\!\!{\begin{tikzpicture}[scale = .15, baseline={(0,0)}] \draw (0,-.5)--(0,0)--(.5,.5) (0,0)--(-.5,.5); \draw (.4,0.8) node[left]{$\scriptstyle#1$}; \draw (-.3,0.8) node[right]{$\scriptstyle#2$}; \end{tikzpicture}}\!\!\!}
\newcommand\smtrecb[2]{{\begin{tikzpicture}[scale = .15, baseline={(0,0)}] \draw (0,-.5)--(0,0)--(.5,.5) (0,0)--(-1.5,1.5) (-1,1)--(-.5,1.5); \draw (-1.3,1.8) node[right]{$\scriptstyle#1$}; \draw (-.3,.8) node[right]{$\scriptstyle#2$}; \end{tikzpicture}}}
\newcommand\smtoverrightcomb[2]{{\!\!\begin{tikzpicture}[scale = .15, baseline={(0,0.1)}] \draw (0,-.5)--(0,0)--(2.5,2.5) (0,0)--(-.5,.5) (2,2)--(1.5,2.5); \draw[fill] (0,.7) circle (.07) (.6,1.3) circle (.07) (1.2,1.9) circle (.07);
\draw (-1.3,1) node{$\scriptstyle#1$} (.7,3) node{$\scriptstyle#2$}; \end{tikzpicture} }}

\newcommand\tO{{\begin{tikzpicture}[scale = .3, baseline={(0,-0.1)}] \draw (0,-.5)--(0,0)--(.5,.5) (0,0)--(-.5,.5); \end{tikzpicture}}}
\newcommand\tA{{\begin{tikzpicture}[scale = .25, baseline={(0,0)}] \draw (0,-.5)--(0,0)--(-1,1)--(-1.5,1.5) (0,0)--(.5,.5) (-1,1)--(-.5,1.5); \end{tikzpicture}}}
\newcommand\tB{{\begin{tikzpicture}[scale = .25, baseline={(0,0)}] \draw (0,-.5)--(0,0)--(1,1)--(1.5,1.5) (0,0)--(-.5,.5) (1,1)--(.5,1.5); \end{tikzpicture}}}
\newcommand\tC{{\begin{tikzpicture}[scale = .2, baseline={(0,0.1)}] \draw (0,-.5)--(0,0)--(-1,1)--(-2,2)--(-2.5,2.5) (0,0)--(.5,.5) (-1,1)--(-.5,1.5) (-2,2)--(-1.5,2.5);  \end{tikzpicture}}}
\newcommand\tD{{\begin{tikzpicture}[scale = .2, baseline={(0,0.1)}] \draw (0,-.5)--(0,0)--(-1,1)--(-1.5,1.5) (0,0)--(.5,.5) (-1,1)--(0,2)--(.5,2.5) (0,2)--(-.5,2.5); \end{tikzpicture}}}
\newcommand\tE{{\begin{tikzpicture}[scale = .2, baseline={(0,0)}] \draw (0,-.5)--(0,0)--(-1,1)--(-1.5,1.5) (-1,1)--(-.5,1.5) (0,0)--(1,1)--(1.5,1.5) (1,1)--(.5,1.5); \end{tikzpicture}}}
\newcommand\tF{{\begin{tikzpicture}[scale = .2, baseline={(0,0.1)}] \draw (0,-.5)--(0,0)--(1,1)--(1.5,1.5) (1,1)--(-.5,2.5) (0,0)--(-.5,.5) (0,2)--(.5,2.5); \end{tikzpicture}}}
\newcommand\tG{{\begin{tikzpicture}[scale = .2, baseline={(0,0.1)}] \draw (0,-.5)--(0,0)--(1,1)--(2,2)--(2.5,2.5) (0,0)--(-.5,.5) (1,1)--(.5,1.5) (2,2)--(1.5,2.5); \end{tikzpicture}}}
\newcommand\trightcomb[1]{{\begin{tikzpicture}[scale = .25, baseline={(0,0.1)}] \draw (0,-.5)--(0,0)--(2.5,2.5) (0,0)--(-.5,.5) (2,2)--(1.5,2.5);
\draw (.35,1.45) node{\rotatebox[origin=c]{80}{$\ddots$}}; \draw (0,1.9) node{$#1$}; \end{tikzpicture}}}
\newcommand\smtrightcomb[1]{{\!\!\begin{tikzpicture}[scale = .15, baseline={(0,0.1)}] \draw (0,-.5)--(0,0)--(2.5,2.5) (0,0)--(-.5,.5) (2,2)--(1.5,2.5); \draw[fill] (0,.7) circle (.07) (.6,1.3) circle (.07) (1.2,1.9) circle (.07); \draw (0,1.9) node{$\scriptstyle#1$}; \end{tikzpicture}}}
\newcommand\smtO{{\begin{tikzpicture}[scale = .2, baseline={(0,-0.1)}] \draw (0,-.5)--(0,0)--(.5,.5) (0,0)--(-.5,.5); \end{tikzpicture}}}

\newcommand\pO{|}
\newcommand\pA{{\begin{tikzpicture} \draw (0,.3)--(0,0) (.3,.3)--(.3,0); \end{tikzpicture}}}
\newcommand\pB{{\begin{tikzpicture} \draw (0,.3)--(0,0)--(.3,0)--(.3,.3); \end{tikzpicture}}}
\newcommand\pC{{\begin{tikzpicture} \draw (0,.3)--(0,0) (.3,0)--(.3,.3) (.6,0)--(.6,.3); \end{tikzpicture}}}
\newcommand\pD{{\begin{tikzpicture} \draw (0,.3)--(0,0)--(.3,0)--(.3,.3) (.6,0)--(.6,.3); \end{tikzpicture}}}
\newcommand\pE{{\begin{tikzpicture} \draw (0,.3)--(0,0) (.3,.3)--(.3,0)--(.6,0)--(.6,.3); \end{tikzpicture}}}
\newcommand\pF{{\begin{tikzpicture} \draw (0,.3)--(0,0)--(.6,0)--(.6,.3)  (.3,.3)--(.3,.1); \end{tikzpicture}}}
\newcommand\pG{{\begin{tikzpicture} \draw (0,.3)--(0,0)--(.6,0)--(.6,.3)  (.3,.3)--(.3,0); \end{tikzpicture}}}
\newcommand\pH{{\begin{tikzpicture} \draw (0,.3)--(0,0)--(.6,0)--(.6,.3)  (.3,.3)--(.3,.1) (.8,0)--(.8,.3); \end{tikzpicture}}}

\newcommand\partexB{\begin{tikzpicture} \draw (-.3,.3) -- (-.3,0) (0,.3) -- (0,0) -- (.6,0) -- (.6,.3) (.3,.3) -- (.3,0.1); \end{tikzpicture}}
\newcommand\partexC{\begin{tikzpicture} \draw (-.3,.3) -- (-.3,0) (0,.3) -- (0,0) -- (.3,0) -- (.3,.3) (.6,.3) -- (.6,0) (.9,.3) -- (.9,0) -- (1.5,0) -- (1.5,.3) (1.2,.3) -- (1.2,0.1); \end{tikzpicture}}
\newcommand\partexD{\begin{tikzpicture} \draw (-.3,.3) -- (-.3,0) (0,.3) -- (0,0) -- (.3,0) -- (.3,.3) (.6,.3) -- (.6,.1) (.9,.3) -- (.9,0) -- (1.5,0) -- (1.5,.3) (1.2,.3) -- (1.2,0.1) (.3,0) -- (.9,0); \end{tikzpicture}}
\newcommand\partexE{\begin{tikzpicture} \draw (-.3,.3) -- (-.3,0) (0,.3) -- (0,0) -- (.3,0) -- (.3,.3) (.6,.3) -- (.6,0); \end{tikzpicture}}

\theoremstyle{plain}
\newtheorem{thm}{Theorem}[section]
\newtheorem{prop}[thm]{Proposition}
\newtheorem{lem}[thm]{Lemma}
\newtheorem{coroll}[thm]{Corollary}

\theoremstyle{definition}
\newtheorem{defn}[thm]{Definition}

\newtheorem{nota}[thm]{Notation}
\newtheorem{ex}[thm]{Example}
\newtheorem{exs}[thm]{Examples}

\theoremstyle{remark}
\newtheorem{rmk}[thm]{Remark}

\title{Planar binary trees, noncrossing partitions \\ and the operator-valued S-transform }

\author[Kurusch Ebrahimi-Fard]{Kurusch Ebrahimi-Fard${}^{\diamond}$}
\address[${}^{\diamond}$]{Department of Mathematical Sciences, Norwegian University of Science and Technology (NTNU), 7491 Trondheim, Norway. Centre for Advanced Study (CAS), 0271 Oslo, Norway.}
\email{kurusch.ebrahimi-fard@ntnu.no}
\urladdr{https://folk.ntnu.no/kurusche/}

\author[Timoth\'e Ringeard]{Timoth\'e Ringeard${}^{\dagger}$}
\address[${}^{\dagger}$]{\'Ecole Normale Sup\'erieure, 75005 Paris, France.}
\email{timothe.ringeard@ens.psl.eu}


\date{\today}

\begin{document}

\begin{abstract}
We revisit the twisted multiplicativity property of Voiculescu's S-transform in the operator-valued setting, using a specific bijection between planar binary trees and noncrossing partitions. 
\end{abstract}

\maketitle

\tableofcontents


\section{Introduction}
\label{sec:intro}

Operator-valued free probability theory expands upon Voiculescu's free probability by incorporating expectation maps that are algebra-valued. Essential concepts from scalar-valued theory, such as R- and S-transforms, have been extended to the operator-valued context. The mathematical foundation for exploring operator-valued free probability involves the concept of formal multilinear function series defined over an operator-valued probability space, as detailed in \cite{dykema2006strans,dykema2007strans}. The reader is refereed to Mingo and Speicher \cite{MingoSpeicher2017} and Speicher \cite{speicher_notes_op_valued} for background on operator-valued free probability. It is interesting to remark that other types of non-commutative independence also exist (monotone and boolean), which fit naturally in the same combinatorial framework, see for instance \cite{bengsch, monotone_free_bool_cum_shuffle,muraki}.

Following Speicher's seminal work \cite{speicher1994multiplicative}, noncrossing set partitions emerge as the main combinatorial tool in -- operator-valued -- free probability.  It is worth recalling that both the count of distinct noncrossing partitions for a set with n elements and the number of planar binary trees with n internal nodes are intricately tied to the Catalan numbers $C_{n}$. Given this connection, it becomes natural to explore one-to-one correspondences that capture the inherent structural similarities between these two combinatorial entities. In light of these considerations, the primary objective of this work is to use a specific bijection between planar binary trees and noncrossing set partitions to prove two central formulas within the realm of operator-valued free probability: the twisted multiplicativity of the S-transform in Theorem \ref{thm:S_transf_mult_free_proba}, and a factorisation identity in Proposition \ref{prop:formula_cum_product_free_vars}. We note that the multiplicativity of the operator-valued S-transform was shown by Dykema using different methods \cite{dykema2006strans,dykema2007strans}, and the factorisation identity is known and proven in the scalar-valued case, see for instance the standard reference by Nica and Speicher \cite{nica_speicher_book}. We provide a proof for an analog statement in the operator-valued case. This approach not only highlights the interplay between combinatorial structures but also provides a powerful tool for deriving and understanding key results in the operator-valued free probability framework. 
 
More precisely, considering a product $ab$ of two freely independent random variables, we revisit Dykema's \cite{dykema2006strans,dykema2007strans} twisted multiplicativity formula for Voiculescu's operator-valued S-transform \cite{VDN1992,voiculescu1995} 
$$
	S_{ab} = S_{b} \cdot S_{a}\circ (S_{b}^{-1}\cdot  I \cdot  S_{b}),
$$
using planar binary trees instead of noncrossing partitions. Here $S_a$ and $S_b$ are algebra-valued formal multilinear function series. The latter can be equipped with compatible noncommutative multiplication and composition -- as seen on the righthand side of the above identity. We note that in the scalar-valued, that is, commutative setting, the term $S_{b}^{-1}\cdot  I \cdot  S_{b}$ simplifies to the identity map $I$ implying that the scalar-valued S-transform of a product of free random variables factorizes, $S_{ab} = S_{b}S_{a}= S_{a} S_{b}$. 

Considering the sum $a+b$ of two free random variables, it is natural to consider the R-transform which is defined in terms of cumulants and which linearizes free convolution in the sense that
$$
	R_{a+b} = R_{a} + R_{b}.
$$
In the scalar-valued setting, the R-transform with respect to a product $ab$ of two freely independent random variables can be understood at the level of cumulants through the well-known formula \cite{nica_speicher_book} 
\begin{equation}
\label{eq:scalar_valued_cumulant_of_ab}
	c_{n}(ab,\ldots,ab) = \sum_{\pi \in \mathrm{NC}(n)} c_{\pi}(a) c_{{\rm Kr}(\pi)}(b),
\end{equation}
which expresses the n-th cumulant of the product $ab$ in terms of a sum of products of cumulants with respect to $a$ and $b$ alone; Speicher showed that freeness of random variables is equivalent to the vanishing of mixed free cumulants. Here ${\rm Kr}(\pi) \in \mathrm{NC}(n)$ is the so-called Kreweras complement of the noncrossing partition $\pi \in \mathrm{NC}(n)$. In the same spirit as for the S-transform, we will revisit the scalar-valued identity \eqref{eq:scalar_valued_cumulant_of_ab} by proving, in Proposition \ref{prop:formula_cum_product_free_vars}, its operator-valued analog using the refined combinatorial structure of planar binary trees. 

A noteworthy observation pertains to the specific bijection between binary trees and noncrossing partitions, elaborated in the following section. We propose that, throughout the paper, exploring the application of this bijection can offer insights into the results concerning the corresponding noncrossing partitions. It serves as a valuable perspective, highlighting instances where the explicit recursive nature of binary trees proves advantageous compared to relying solely on noncrossing partitions. This comparison is detailed in the first section. Additionally, it could be valuable to explore Knuth's rotation correspondence, which maps planar binary trees to planar rooted trees, providing another avenue for investigation.

\medskip

Here is an outline of the paper. In Section \ref{sec:bij_tree_non_crossing_partition}, a bijection between noncrossing partitions and planar binary trees is given, that will be useful in the rest of the paper. A general way of constructing such a bijection, based on the concept of Catalan pairs, is presented. In Appendix \ref{app:other_bij} it is shown how this concept can be used to describe other already existing bijections between trees and noncrossing partitions. Section \ref{sec:S_transf_multilin_fct_series} explores formal series of multilinear functions and how they can be constructed using binary trees. Several so-called boxed convolution products on such series are defined together with S-transforms thereof using binary trees. In Appendix \ref{app:operad_pt_view} an operadic version of several of the results of Section \ref{sec:S_transf_multilin_fct_series} are presented which complement the combinatorial proofs therein. In Section \ref{sec:meaning_free_proba}, it is explained how the results can be interpreted in the context of operator-valued free probability. In particular, we give a meaning to one of the boxed convolutions, which corresponds to the multiplication of two free non-commutative random variables. The results presented in the previous sections allow to re-derive Dykema's twisted factorisation property of Voiculescu's S-transform. We also provide a operator-valued formula for free cumulants of the product of two free random variables, which depends on the free cumulants of the respective random variables (see Proposition \ref{prop:formula_cum_product_free_vars}). 

\medskip

{\bf{Acknowledgements}}: KEF is supported by the Research Council of Norway through project 302831 “Computational Dynamics and Stochastics on Manifolds” (CODYSMA). He would also like to thank the Centre for Advanced Study (CAS) in Oslo for hospitality and its support.  TR would like to thank the Department of Mathematical Sciences at the Norwegian University of Science and Technology in Trondheim for hospitality during a visit in the spring of 2023.


\section{Catalan pairs, binary trees and noncrossing partitions}
\label{sec:bij_tree_non_crossing_partition}

Planar binary trees and noncrossing partitions have an interesting combinatorial connection. Indeed, there exist many bijections between them, and these bijections preserve some kind of structure on both sides (see for example \cite{edelman_bij}, or \cite{prodinger}, \cite{dershowitz_zaks} along with Knuth's correspondance \cite{knuth}). We are going to define yet another bijection, by defining a Catalan structure on both planar binary trees and noncrossing partitions. This bijection will preserve the Catalan structure. The motivation is that using the former -- more obvious and visual -- Catalan structure can be helpful in computations involving the latter together with its more hidden Catalan structure.

For the reader wanting to learn more about the relations between several Catalan objects, we refer the reader to the recent book by Stanley \cite{Stanley2015}. We remark that the use of trees in the study of noncrossing partitions appears for instance in a standard part of the analysis of GUE matrices, see \cite[2.1.3, 2.1.4]{anderson}.


\subsection{Catalan pairs}
\label{ssec:catalan_pairs}

The common denominator that we are going to use between planar binary trees and noncrossing partitions is the notion of Catalan pair, defined next. We remark that the terminology of Catalan pairs also appears in \cite{disanto_catalan_pairs}, but is unrelated to the notion we introduce in this paper.

\begin{defn}
\label{def:isom_Cat_sets}
Let $C = (C_n)_{n\ge 0}$ be a collection of finite sets, and assume that $C_0 = \{1_C\}$. A \textit{Catalan pair} $(C,f)$ over $C$ is defined in terms of a \textit{Catalan map} $f:C\times C\to C$ inducing bijections $\bigsqcup_{k=0}^{n} \big(C_k \times C_{n-k}\big) \to C_{n+1}$ for all $n\ge 0$. 
\end{defn}

\begin{rmk}
\begin{enumerate}
	\item If $(C,f)$ is a Catalan pair, then $|C_n| = \frac 1 {n+1} \binom{2n}{n}$ is the $n$-th Catalan number.
	
	\item A Catalan map $f$ gives a way to construct all elements of $C$ from the element of $C_0=\{1_C\}$. Indeed, from Definition \ref{def:isom_Cat_sets} it is clear that any element in $\bar{C}:=(C_n)_{n > 0}$ can be written as iterated expressions of the map $f$ and $1_C$ such as, for example, $f(1_C, 1_C)$ or $f( f(1_C, f(1_C,1_C)),~ f(1_C, 1_C))$. Moreover, these expressions are unique. Both of these facts follow from simple inductions.
	
	\item Inductions are particularly efficient in the context of Catalan pairs: if a property is true for the element $1_C$ and preserved by the Catalan map $f$ (in the sense that if $a,b$ satisfy both the property so does $f(a,b)$), then the property is true for all elements.
\end{enumerate}
\end{rmk}

\begin{lem}
\label{lem:isom_Cat_pairs}
Let $(C,f)$ and $(C',f')$ be Catalan pairs. Then there exists a unique isomorphism of Catalan pairs, i.e., there exists a unique bijection $\varphi: C\to C'$ such that $\varphi(1_C) = 1_{C'}$ and for all $x, y\in C$, $\varphi (f(x,y)) = f'(\varphi(x), \varphi(y))$.
\end{lem}

\begin{proof}
The map $\varphi$ can be constructed on $C_n$ by induction on $n$, as does unicity.
\end{proof}

\begin{lem}
\label{lem:cat_pair_other_decomp}
Let $(C,f)$ be a Catalan pair. Then for $n\ge 0$:
\begin{align*}
	C_n &= \bigsqcup_{\substack{n=k_1+\cdots+k_m \\ k_i \ge 1}}\{f(a_1, f(a_2, f(\ldots, f(a_m, 1_C))) \mid a_i \in C_{k_i-1},\ 1\le i\le m\}.
\end{align*}
\end{lem}

\begin{proof}
Let
$$
	g_n : (a_1, \ldots, a_m) \in \bigsqcup_{\substack{n=k_1+\cdots+k_m \\ k_i \ge 1}} 
	C_{k_1-1}\times \cdots\times C_{k_m-1} \mapsto f(a_1, f(a_2, f(\ldots, f(a_m, 1_C))) \in C_n .
$$
Then $g_n$ is injective: $f(a_1, f(a_2, f(\ldots, f(a_m, 1_C))) = f(b_1, f(b_2, f(\ldots, f(b_{m'}, 1_C)))$, then by injectivity of $f$, $a_1=b_1$ and $f(a_2, f(\ldots, f(a_m, 1_C))) = f(b_2, f(\ldots, f(b_{m'}, 1_C)))$ and by induction we arrive to $m=m'$ and $a_i=b_i$ for all $1\le i\le m$. Next, we  show by induction on $n$ that $g_n$ is surjective. Let $a \in C_n$. Then $a = f(a_1, a')$ with $a_1 \in C_{k_1-1}$, $a' \in C_{n'}$, $k_1\ge 1$, $n' \ge 0$ and $n = k_1+n' > n'$. So we can write by induction $a' = f(a_2, f(\ldots, f(a_m, 1_C)))$, and thus $a = f(a_1, f(a_2, f(\ldots, f(a_m, 1_C)))$.
\end{proof}


\subsection{Planar binary trees and noncrossing partitions}
\label{ssec:bin_trees_ncp}

Before defining the Catalan pair structures on binary trees and noncrossing partitions, let us first recall some notations for planar binary trees.

\begin{defn}
\label{def:Y_trees}
\begin{enumerate}
\item A \textit{planar binary tree} is a rooted tree where each internal vertex has exactly two children, a right one and a left one. More formally, for $n \ge 0$, the set of planar binary trees with $n$ internal vertices is
\allowdisplaybreaks
\begin{align*}
	Y_0 &= \{|\} \\
	Y_n &= \bigcup_{\substack{k+l=n-1\\ k,l\ge 0}}\{\tovunder \sigma \tau \mid \sigma \in Y_k,\ \tau \in Y_l\} \qquad n > 0,
\end{align*}
where $|$ is the so-called empty tree, i.e., with no internal vertex. Let $Y := \bigcup_{n\ge 0} Y_n$ and $\overline Y := \bigcup_{n\ge 1} Y_n$. Denoting for a tree $\tau \in Y_n$ the set of internal vertices by $V(\tau)$, we write $|\tau|:=|V(\tau)|=n$. We will abusively call vertex an internal vertex.

\item For a tree $\tau \in \overline Y_n$, the "left-to-right" ordering of its vertices is the linear order $o(\tau)=\{v_1 < \cdots < v_n\}$ on its vertices $v_i \in V(\tau)$, defined recursively by
$$
	o\left( \children{\sigma}{x}{\tau} \right) = \{o(\sigma) < x < o(\tau)\}.
$$
We can therefore identify the vertices of $\tau$ with the set $[|\tau|]:=\{1, \ldots, |\tau|\}$. 
\end{enumerate}
\end{defn}

As an example, we consider the planar binary tree on the lefthand side of Figure 1, where only internal vertices are shown. The "left-to-right" ordering $o(\tau)$ is indicated by the labelling of the vertices of $\tau$. 

\begin{defn}
\label{def:grafttrees}
We define two grafting operations, called {\it{over}} and {\it{under}}, on trees
$$
	/, \backslash : Y\times Y \to Y,
$$ 
where $\tau/\sigma$ (read $\tau$ over $\sigma$) is the tree obtained by adding $\tau$ as the left-child of the left-most internal vertex of $\sigma$, and $\tau \backslash \sigma$ (read $\tau$ under $\sigma$) is the tree obtained by adding $\sigma$ as the right-child of the right-most internal vertex of $\tau$.
\end{defn}

For example:
$$
	\tO / \tO= \tA
	\qquad\
	\tO \backslash \tO = \tB
	\qquad\
	(\tO \backslash \tO ) / \tO = \tD
	\qquad\
	\tO  \backslash (\tO / \tO) = \tF.
$$

\begin{rmk}
\label{emk:operad}
\begin{enumerate}
\item
The grafting operations, {\it{over}} and {\it{under}}, on planar binary trees where introduced by J.-L.~Loday in \cite{loday2002}, which provides a framework for understanding algebraic structures on trees, particularly in the context of dendriform operations, by visualising them using combinatorial structure of planar rooted trees.

\item
Note that in both $\tau$ over $\sigma$ and $\tau$ under $\sigma$, the vertices of $\tau$ appear before the vertices of $\sigma$ in the left-to-right ordering. The words \textit{over} and \textit{under} refer to the tree growing up, rather than the left-to-right order. In that sense, the root of $\tau$ is higher than the root of $\sigma$ in $\tau$ over $\sigma$, for example.

\item
For $\tau\in \overline Y$, we define the substitution $\mu_\tau : \overline Y^{|\tau|} \to \overline Y$, where the tree $\mu_\tau(\sigma_1, \ldots, \sigma_{|\tau|})$ is obtained from the tree $\tau$ by replacing each vertex $v_i \in V(\tau)$ with the tree $\sigma_i$. Note that the order on the vertices is $o(\mu_\tau(\sigma_1, \ldots, \sigma_{|\tau|})) = \{o(\sigma_1) < \cdots < o(\sigma_{|\tau|})\}$, where we identify the vertices of $\sigma_i$ with their image in $\mu_\tau(\sigma_1, \ldots, \sigma_{|\tau|})$. We remark that this operation defines a operadic structure on $\overline Y$(in the sense of a non-symmetric set-operad, as opposed to a symmetic or linear operad), which allows to write $\mu_\tau(\sigma_1, \ldots, \sigma_{|\tau|})$ in operadic notation
\begin{equation}
\label{eq:operadprod}
	\tau \circ (\sigma_1, \ldots, \sigma_{|\tau|}) \in  \overline Y.
\end{equation}
We refer the reader to Frabetti's work \cite{frabetti} for more details. 
\end{enumerate}
\end{rmk}

The definition of planar binary trees implies the following statement.

\begin{lem}
\label{lem:Cat_pair_Y}
Define the map $\mathrm{Cat}_Y: Y\times Y \to Y$, $(\sigma, \tau) \mapsto \tovunder \sigma \tau$. 
Then $(Y, \mathrm{Cat}_Y)$ is a Catalan pair.
\end{lem}

\begin{defn}
\label{def:partitions_basics}
\begin{enumerate}
\item For $n\ge 0$, let $\text{SP}_n$ denote the set of partitions on $[n] = \{1, \ldots, n\}$, and $\NCP_n \subseteq \text{SP}_n$ is the set of noncrossing partitions of $[n]$. The latter is a set partition $P = \{V_1,\ldots,V_k\}$ of $[n]$ without elements $a<b<c<d$ in $[n]$ such that $a,c \in V_i$ and $b,d \in V_j$, and $i \neq j$. We call $P \in \NCP_n$ irreducible if both $1$ and $n$ belong to the same block in $P$. Let $\text{SP} := \bigcup_{n\ge 0} \text{SP}_n$ and $\NCP := \bigcup_{n\ge 0} \NCP_n$. By convention $\NCP_0 = \{\emptyset\}=\text{SP}_0$ consists of the empty partition. If $P \in \NCP_n$, we write $|P| = n$. For $n\ge 1$, we will denote by $0_n= \{ \{1\},\ldots, \{n\}\}$ the partition of $[n]$ into $n$ singletons, i.e., blocks of size one, and by $1_n= \{\{1,\ldots,n\}\}$ the one-block partition with $n$ elements. The standard diagrammatic representation for noncrossing partitions is used, where, for instance, $0_n = | | | \cdots |$ and $1_n = {\begin{tikzpicture}[baseline={(0,0.05)}] \draw (0,.3)--(0,0)--(1,0)--(1,.3) (.2,.3)--(.2,0) (.8,.3)--(.8,0); \draw (.5,.2) node{$\cdots$}; \end{tikzpicture}}$. 

\item For $P,Q \in \NCP$, define $P * Q$ to be the noncrossing partition resulting from concatenation of the two partitions.

\item For $P,Q \in \NCP$, define the right-merging $P \underline{*} Q$ to be the noncrossing partition obtained from $P * Q$ by merging the last block of $P$ and the last block of $Q$ into a single block. 
\end{enumerate}
\end{defn}

\begin{rmk}
\label{rmk:associativity}
\begin{enumerate}
\item
Observe that both concatenation and right-merging of partitions are associative but non-commutative operations. As an example for $P * Q$ and $P \underline{*} Q$, consider the two partitions $\pE$ and $\partexB $\ , then
$$
	\pE ~*~ \partexB \ =\ \partexC \quad\quad\quad \pE ~\underline*~ \partexB \ =\ \partexD\ .
$$

\item The two operations of concatenation and right-merging also satisfy the following mixed-associative identity. For partitions $P,Q$ and $R$, we have
\[
	(P * Q) \underline* R = P * (Q \underline* R).
\]
Indeed, the partition obtained in both cases amounts to concatenating $P,Q$ and $R$, and then connecting the last block of $Q$ with the last block of $R$. Note, however, that the reversed identity does not hold, i.e. $(P \underline* Q) * R \neq P \underline* (Q * R)$ in general.

\end{enumerate}
\end{rmk}

\begin{lem}
\label{lem:Cat_pair_NCP}
Define the map $\mathrm{Cat}_\mathrm{NCP}: \mathrm{NCP} \times \mathrm{NCP} \to \mathrm{NCP}$, 
\begin{equation}
\label{eq:NCPcatalanpair}
	(P,Q) \mapsto P*(|\underline * Q).
\end{equation}
Then $(\mathrm{NCP}, \mathrm{Cat}_\mathrm{NCP})$ is a Catalan pair.
\end{lem}
\begin{proof}
This amounts to saying that every noncrossing partition $R$ has a unique decomposition $R=P*(|\underline * Q)$, as depicted next
\begin{figure}[h!]
\begin{tikzpicture}
	\draw (0,0) rectangle (2,1) (3,0) rectangle (5,1);
	\draw (2.5,1) -- (2.5,-0.5) -- (5,-0.5) -- (5,0);
	\draw (-1,0.5) node{$R~~=$} (1,0.5) node{$P$} (4,0.5) node{$Q$};
\end{tikzpicture}~.
\end{figure}

For example, for $R = ~\partexC~$, this decomposition is $P = ~\partexE~$ and $Q = ~\pA~$.
\end{proof}

Combining this with Lemma \ref{lem:Cat_pair_Y} and  Lemma \ref{lem:isom_Cat_pairs} gives:

\begin{thm}
\label{def:varphi}
There exists a unique isomorphism $\varphi : Y \to \mathrm{NCP}$ of Catalan pairs.
\end{thm}

\begin{rmk}
\label{def:varphi_explicit}
We give an explicit description of the bijection $\varphi : Y \to \NCP$, which is graded in the sense that $\varphi(Y_n) = \NCP_n$, and was already considered in \cite{novelli_thibon_varphi}. Let us construct $\varphi(\tau)$ for $\tau\in Y_n$, $n\ge 1$. First, number the vertices of $\tau$ by elements from $[n]$ using the left-to-right order. Then $\varphi (\tau)\in \NCP(n)$ is the noncrossing partition on $[n]$ where the block of $i\in [n]$ is the set of all vertices connected to $i$ by "right arms". More precisely, define $R(x,y)$ to be the binary relation on $[n]$: "x is the right child of y". Then $\varphi (\tau)$, seen as an equivalence relation, is the reflexive-symmetric-transitive closure of $R$. Figure 1 shows an example, and Example \ref{ex:first_value_varphi} shows the first few values of $\varphi$ on small trees.

To verify that the constructed map $\varphi$ is the same as the one in Definition \ref{def:varphi}, it suffices to notice that $\varphi(|) = \emptyset$, and that for all trees $\sigma, \tau$ in $Y$, the inductive formula holds
\begin{equation}
\label{eq:inductive}
	\varphi(\tovunder \sigma \tau) = \varphi (\sigma) * \big(| \underline * \varphi(\tau)\big).
\end{equation}
\end{rmk}

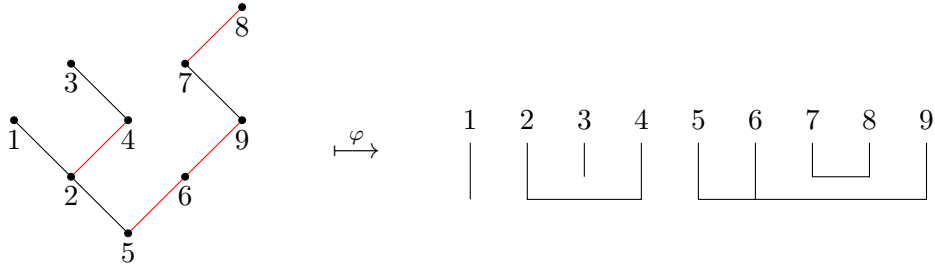
\begin{figure}[ht!]
\label{fig:ex_of_varphi}
\begin{center}
\begin{tikzpicture}[scale = 1.5]
\draw[red] (0,0) -- (.5,.5) -- (1,1) (.5,1.5) -- (1,2) (-.5,.5) -- (0,1);
\draw (1,1) -- (.5,1.5) (0,0) -- (-.5,.5) -- (-1,1)  (0,1) -- (-.5,1.5);
\draw[fill] (-1,1) circle (.03) node[below]{1};
\draw[fill] (-.5,.5) circle (.03) node[below]{2};
\draw[fill] (-.5,1.5) circle (.03) node[below]{3};
\draw[fill] (0,1) circle (.03) node[below]{4};
\draw[fill] (0,0) circle (.03) node[below]{5};
\draw[fill] (.5,.5) circle (.03) node[below]{6};
\draw[fill] (.5,1.5) circle (.03) node[below]{7};
\draw[fill] (1,2) circle (.03) node[below]{8};
\draw[fill] (1,1) circle (.03) node[below]{9};

\draw (2,.8) node{$\overset{\varphi}{\longmapsto}$}; 

\draw (3, 1) node{1};
\draw (3.5, 1) node{2};
\draw (4, 1) node{3};
\draw (4.5, 1) node{4};
\draw (5, 1) node{5};
\draw (5.5, 1) node{6};
\draw (6, 1) node{7};
\draw (6.5, 1) node{8};
\draw (7, 1) node{9};

\draw (3, .8) -- (3, 0.3);
\draw (3.5, .8) -- (3.5, 0.3) -- (4.5, 0.3) -- (4.5, .8);
\draw (4, .8) -- (4, 0.5);
\draw (5, .8) -- (5, 0.3) -- (7, 0.3) -- (7, .8) (5.5, 0.3) -- (5.5, .8);
\draw (6, .8) -- (6, 0.5) -- (6.5, 0.5) -- (6.5, .8);
\end{tikzpicture}
\caption{An example value of the function $\varphi$. Only internal vertices are shown in the tree on the left.}
\end{center}
\end{figure}

\begin{ex}
\label{ex:first_value_varphi}
Here are the first values of $\varphi$:
\begin{align*}
	\varphi\left(\tO\right) &= \pO & \varphi\left(\tD\right) &= \pD\\
	\varphi\left(\tA\right) &= \pA & \varphi\left(\tE\right) &= \pE\\
	\varphi\left(\tB\right) &= \pB & \varphi\left(\tF\right) &= \pF\\
	\varphi\left(\tC\right) &= \pC & \varphi\left(\tG\right) &= \pG
\end{align*}
\end{ex}

The following lemma shows the relation between the inductive formula \eqref{eq:inductive} and the two grafting operations on trees, over and under, given in Definition \ref{def:grafttrees}.

\begin{lem}
\label{lem:varphi_over_under}
For $\sigma, \tau\in Y$, 
\begin{align}
	\varphi(\sigma/\tau) 
		&= \varphi(\sigma) * \varphi(\tau) \\
	\varphi(\sigma\backslash \tau) 
		&= \varphi(\sigma) \underline* \varphi(\tau).
\end{align}
\end{lem}

\begin{proof}
The first equality can be shown by induction on the size of $\tau$. If $\tau = |$, the equality is clear. Otherwise, let us write $\tau = \tovunder {\tau_1}{\tau_2}$ with $|\tau_1|<|\tau|$. Then $\sigma/\tau = \tovunder{\sigma/\tau_1}{\tau_2}$ so
\allowdisplaybreaks
\begin{align*}
	\varphi(\sigma/\tau) 
	&= \varphi(\tovunder{\sigma/\tau_1}{\tau_2}) \\
	&\stackrel{\eqref{eq:inductive}}{=} \varphi(\sigma/\tau_1) *( | \underline * \varphi(\tau_2)) \\
	&\overset{\text{ind.}}{=} \varphi(\sigma)*\varphi(\tau_1) *( | \underline * \varphi(\tau_2)) \\
	&= \varphi(\sigma) * \varphi(\tovunder{\tau_1}{\tau_2}) \\
	&= \varphi(\sigma) * \varphi(\tau).
\end{align*}
The second equality follows by a symmetry argument.
\end{proof}

\begin{rmk}
We will not provide an explicit description of the inverse, $\varphi^{-1}$, and only note that, as for all isomorphisms between Catalan pairs, there is an inductive way to construct it. Indeed, going back to \eqref{eq:NCPcatalanpair} in Lemma \ref{lem:Cat_pair_NCP}, we see that given a noncrossing partition $P \in \text{NCP}$, we can write $P = P_1 * (| \underline * P_2)$, and thus 
$$
	\varphi^{-1}(P) = \tovunder{\varphi^{-1}(P_1)}{\varphi^{-1}(P_2)}.
$$
\end{rmk}


\subsection{The Kreweras complement as a Catalan isomorphism}
\label{ssec:Kreweras_complement}

\begin{defn}
\begin{enumerate}
	\item For noncrossing partitions $P,Q\in \NCP$, define $P\, \overline * \, Q$ to be obtained from $P*Q$, by merging the first block of $P$ with the first block of $Q$ into a single block. 

	Define $\text{Cat}'_\NCP : (P,Q)\in \NCP^2 \mapsto P\, \overline *\, | * Q$, which turns $(\NCP, \text{Cat}'_\NCP)$ into a Catalan pair. Define $\psi : Y \to \NCP$ the corresponding isomorphism of Catalan pairs, and $K' : \NCP \to \NCP$ the isomorphism of Catalan pairs $(\NCP, \text{Cat}'_\NCP) \to (\NCP, \text{Cat}_\NCP)$.

	\item (Poset structure of $\NCP_n$ \cite{nica_speicher_book}) We can define an partial order on $\NCP_n$ by reverse refinement: One has $P\le Q$ if every block of $P$ is included in a block of $Q$. This defines a lattice structure on $\NCP_n$, with minimum element $0_n$ and maximum element $1_n$.
	
	\item For partitions $P,Q\in \text{SP}_n$, $n\ge 1$, we define the partition $P \cup Q\in \text{SP}_{2n}$ by partitioning the odd integers through $P$, and the even integers through $Q$.
\end{enumerate}
\end{defn}

\begin{exs}
\begin{enumerate}
	\item We have $~\pF~ \cup ~\pE~ = ~\begin{tikzpicture}[scale=.3, baseline={(0,0)}]
	\draw (0,1)--(0,0)--(4,0)--(4,1) (1,1)--(1,.3) (2,1)--(2,.3) (3,1)--(3, -.4)--(5,-.4)--(5,1);
	\end{tikzpicture}$. Note that the two partitions, $\pF$ and $\pE$, are noncrossing, but not $~\pF~ \cup ~\pE$.
\end{enumerate}

\end{exs}

\begin{defn}(Kreweras complement)
Given $P\in \NCP_n$, the \textit{Kreweras complement} of $P$, denoted by $K(P)$, is defined to be the largest partition $Q\in \NCP_n$ such that $P\cup Q$ is noncrossing.
\end{defn}

\begin{ex}
Consider the noncrossing partition in $\NCP_8$
$$
	P := 
\begin{tikzpicture}[scale = .5, baseline={(0,.1)}]
\draw (0,1)--(0,0)--(1,0)--(1,1) (2,1)--(2,0)--(7,0)--(7,1) (5,0)--(5,1) (3,1)--(3,.3) (4,1)--(4,.3) (6,1)--(6,.3);
\end{tikzpicture}\, . 
$$
Then its Kreweras complement
$$
	K(P) =
\begin{tikzpicture}[scale = .5, baseline={(0,.1)}]
\draw (0,1)--(0,0) (1,1)--(1,0)--(7,0)--(7,1) (2,1)--(2,.3)--(4,.3)--(4,1) (3,1)--(3,.3) (5,1)--(5,.3)--(6,.3)--(6,1);
\end{tikzpicture}\, .
$$
\end{ex}

The Kreweras complement has a nice diagrammatical interpretation, depicted in the following example. As we will be particularly interested in the noncrossing partition $P\cup K(P)$, it will be useful to have the following notation: Interlace $\tilde 1, \ldots, \tilde n$ with $1, \ldots, n$ in the following way:
$$
	1, \tilde 1, 2, \tilde 2, \ldots, n, \tilde n.
$$
Then the noncrossing partition $P\cup K(P)$ can be thought of as a partition over $1, \tilde 1, 2, \tilde 2, \ldots, n, \tilde n$. For example, if $P$ is the partition of the last example, then the partition $P\cup K(P)$ is:
$$
	
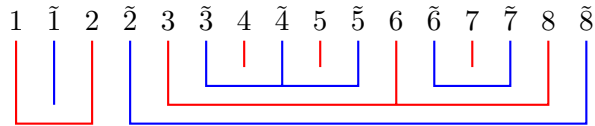
\begin{tikzpicture}[scale = .5]
	\draw[thick, red] (0,1)--(0,-1.2)--(2,-1.2)--(2,1) (4,1)--(4,-.7)--(14,-.7)--(14,1) (10,-.7)--(10,1) (6,1)--(6,.3) (8,1)--(8,.3) (12,1)--(12,.3);
	\draw[thick, blue] (1,1)--(1,-.7) (3,1)--(3,-1.2)--(15,-1.2)--(15,1) (5,1)--(5,-.2)--(9,-.2)--(9,1) (7,1)--(7,-.2) (11,1)--(11,-.2)--(13,-.2)--(13,1);
	\draw (0,1) node[above]{$1$} (1,1) node[above]{$\tilde 1$} (2,1) node[above]{$2$} (3,1) node[above]{$\tilde 2$} (4,1) node[above]{$3$} (5,1) node[above]{$\tilde 3$} (6,1) node[above]{$4$} (7,1) node[above]{$\tilde 4$} (8,1) node[above]{$5$} (9,1) node[above]{$\tilde 5$} (10,1) node[above]{$6$} (11,1) node[above]{$\tilde 6$} (12,1) node[above]{$7$} (13,1) node[above]{$\tilde 7$} (14,1) node[above]{$8$} (15,1) node[above]{$\tilde 8$}; 
	\end{tikzpicture}
$$

\begin{lem}
\label{lem:Kreweras_change_*}
For $P,Q \in \mathrm{NCP}$, the Kreweras complement satisfies
\begin{align*}
	K(P * Q) &= K(P) \underline * K(Q), \\ 
	K(P \overline * Q) &= K(P) * K(Q).
\end{align*}
\end{lem}

\begin{proof}
The proof can be done diagrammatically, and the details are left to the reader.
\end{proof}

\begin{lem}
The isomorphism $K'$ of Catalan pairs and the Kreweras complement $K$ coincide.
\end{lem}

\begin{proof}
The map $K'$ is defined inductively by 
\begin{align*}
	K'(\emptyset) &= \emptyset \\ 
	K'(P \overline * | * Q) &= K'(P) * | \underline * K'(Q) \qquad \text{ for }P,Q \in \NCP.
\end{align*}
It suffices to show that $K$ verifies the same equalities. Note that already $K(\emptyset) = \emptyset$, and also $K(|) = |$. Then for $P,Q\in \NCP$, by Lemma \ref{lem:Kreweras_change_*},
$$
	K(P \overline * | * Q) = K(P) * K(| * Q) = K(P) * | \underline * K(Q).
$$
Regarding these equalities it is important to recall Remark \ref{rmk:associativity}.
\end{proof}

Henceforth, this map will be denoted by $K:\mathrm{NCP} \to \mathrm{NCP}$.


\section{The S-transform of multilinear function series}
\label{sec:S_transf_multilin_fct_series}


\subsection{Series of multilinear functions}
\label{ssec:series_multilin_functions}

We follow the terminology of Dykema in \cite{dykema2007strans}. Let $\mathbb{K}$ denote the base field of characteristic zero over which all algebraic structures are defined.

\begin{defn}
Let $B$ be a $\mathbb{K}$-algebra with unit $1$. A sequence of multilinear functions $f = (f_n)_{n\ge 0}$, $f_n : B^{\otimes n} \to B$ (with $B^{\otimes 0} = \mathbb{K}$) will be called a multilinear series. The set of these multilinear series will be called $\Mult [[B]]$. 
\begin{enumerate}
\item
We can multiply such series by defining for $f,g \in \Mult[[B]]$:
\begin{equation}
\label{multipl}
	(f \cdot g)_n(x_1, \ldots, x_n) := \sum_{k=0}^n f_k(x_1, \ldots, x_k) g_{n-k}(x_{k+1}, \ldots, x_n).
\end{equation}
$(\Mult[[B]], \cdot)$ is a monoid, with unit $1 := (\delta_{n,0}1)_{n\ge 0}$. We will often write $fg$ for $f\cdot g \in \Mult[[B]]$. 

\item
We can compose such series by defining for $f,g \in \Mult[[B]]$ such that $g_0=0$:
\begin{equation}
\label{compo}
	(f\circ g)_n(x_1, \ldots, x_n) 
	:= \sum_{\substack{n = k_1 + \cdots + k_l \\ k_i\ge 1}} f_l(g_{k_1}(x_1, \ldots, x_{k_1}), \ldots, g_{k_l}(x_{n-k_l+1}, \ldots, x_n)).
\end{equation}
The set $\{g\in \Mult[[B]], g_0=0\}$ is a monoid from this product $\circ$, with unit $I := (\delta_{n,1}\Id)_{n\ge 0}$. 
\end{enumerate}
Note that if $f\cdot g = 0$, then either $f=0$ or $g=0$. This is, however, not the case for the composition $\circ$.

\end{defn}

We define now groups, $(G_B^\inv ,\cdot )$ and  $(G_B^\dif , \circ)$, with respect to multiplication \eqref{multipl} respectively composition \eqref{compo} of elements in $\Mult[[B]]$:
\begin{align*}
	G_B^\inv &:= \{f \in \Mult [[B]] \mid f_0\in B^\times\} \subset \Mult[[B]] \\
	G_B^\dif &:= \{f \in \Mult [[B]] \mid f_0=0, f_1\in \GL(B)\} \subset \Mult[[B]],
\end{align*}
where $B^\times$ is the set of invertible elements of $B$. We note that replacing $f_0\in B^\times$ by $f_0=1$ and $f_1\in \GL(B)$ by $f_1=\Id$ would not change much. The set $G_B^\inv $ can be mapped to a subset of $G_B^\dif$, which we define as a third set 
\begin{equation}
\label{Ishiftset}
	G_B^I:= I \cdot G_B^\inv.
\end{equation}
Alternatively,  
\begin{equation}
\label{eq:GBI_alternative_def}
	G_B^I := \{f \in G_B^\dif \mid f_1(1)\in B^\times, \forall n\ge 1, f_n(x_1, x_2, \ldots, x_n) = x_1 f_n(1, x_2, \ldots, x_n)\}.
\end{equation}
For $f \in G_B^\inv$, we will denote by $f^{-1}$ the multiplicative inverse of $f$. The compositional inverse of $f \in G_B^\dif$ will be denoted by $f^{\circ-1}$. Regarding expressions involving both composition and multiplication, $\circ$ and $\cdot$, the former operation will be given precedence, i.e., $f \cdot g \circ h \cdot k = f\cdot (g\circ h) \cdot k$.

\begin{lem}
\label{lem:rightaction}
Composition and multiplication, $\circ$ and $\cdot$, are both associative and non-commutative. Moreover, composition is right-distributive over multiplication, that is, for any $f,g,h \in \mathrm{Mult}[[B]]$
\begin{equation}
\label{rightaction}
	(f \cdot g)\circ h = (f\circ h) \cdot (g\circ h).
\end{equation}
\end{lem}

\begin{proof}
See Dykema \cite{dykema2007strans} (Proposition 2.3, (vi), Page 8), for details.
\end{proof}

\begin{lem}
\label{lem:G_B^I_group}
The set $G_B^I$ forms a group with respect to composition $\circ$.
\end{lem}

\begin{proof}
Let $f, g\in G_B^I \subset G_B^\dif$. As $I = I \cdot 1\in G_B^I$, it suffices to show that $f\circ g$ and $f^{\circ -1}$ (calculated in $G_B^\dif$) are elements of $G_B^I$. Writing $f=IF$ and $g=IG$, for $F,G \in G_B^\mathrm{inv}$, and using \eqref{rightaction} in Lemma \ref{lem:rightaction}, we have
\begin{align*}
	f \circ g = (IF) \circ (IG) 
	&= IG \cdot (F\circ(IG))\\
	&= I\cdot \big(G \cdot (F\circ(IG))\big) \in G_B^I. 
\end{align*}
Moreover, 
$$
	I = f\circ f^{\circ-1} = (IF) \circ f^{\circ-1} = f^{\circ-1}\cdot F\circ f^{\circ-1}.
$$ 
But $F\circ f^{\circ-1} \in G_B^\inv$ since $F\in G_B^\inv$, so 
$$
	f^{\circ-1} = I \cdot (F\circ f^{\circ-1})^{-1} \in G_B^I.
$$
\end{proof}

\begin{defn} (S-transform)
\label{def:S_transform}
Let $f\in G_B^I$. Then, as $f^{\circ-1}\in G_B^I$, we can write 
$$
	f^{\circ-1} = I \cdot S_f,
$$ 
where $S_f \in G_B^\mathrm{inv}$ is called the $S$-transform of $f$.
\end{defn}

\begin{lem}
\label{lem:fixed_point_S_f}
Let $f = IF \in G_B^I$, $F \in G_B^\mathrm{inv}$. Its S-transform is the unique solution to the fixed point equation:
\begin{equation}
\label{Strans}
	S_f = (F\circ (IS_f))^{-1}.
\end{equation}
\end{lem}

\begin{proof}
In the proof of Lemma \ref{lem:G_B^I_group}, it is shown that $f^{\circ-1} = I (F\circ f^{\circ-1})^{-1}$. Therefore 
$$
	S_f = (F\circ f^{\circ-1})^{-1} = (F\circ (IS_f))^{-1}.
$$ 
Regarding unicity, let $g$ be a solution of $g = (F\circ(Ig))^{-1}$. The multilinear functions $g_k$, $k < n$, determine the $j$-linear functions $(F\circ(Ig))_j$ for $j\le n$, which in turn determine the $n$-linear function $((F\circ(Ig))^{-1})_n$, using classical Möbius inversion. So $g$ is indeed uniquely defined by induction.
\end{proof}

Observe that the foregoing lemma also provides formulas for $S_f^{-1}IS_f$ (which will appear later).

\begin{coroll}
\label{cor:other_expr_S_f-1IS_f}
Let $f = IF \in G_B^I$. Then 
$$
	S_f^{-1}IS_f = (FI)\circ (IS_f) = (FI)\circ (IF)^{\circ -1}.
$$
\end{coroll}

\medskip

From one or several multilinear series, one can construct other multilinear maps:

\begin{defn}
\label{def:treenesting}
For $f,g \in \Mult[[B]]$ and a planar binary tree $\tau \in Y_n$, which we write as
$$
	\tau = \toverrightcomb{\tau_1}{\tau_2}{\tau_k}
$$
we define:
 
\begin{enumerate}
\item
\label{item1}
$f_\tau : B^{\otimes n} \to B$ inductively as follows:
$$
	\left\{\begin{array}{rcl}
	f_{|} &=& 1 \in B \\ 
	f_\tau(x_1, \ldots, x_n) &=& f_k\Big(f_{\tau_1}(x_1, \ldots, x_{j_1-1}) x_{j_1}, ~\ldots~, 
	~f_{\tau_k}(x_{j_{k-1}+1}, \ldots, x_{n-1}) x_n \Big) \in B 	
	\end{array}\right.
$$

\item
\label{item2}
$(f\cup g)_\tau : B^{\otimes n} \to B$ inductively as follows:
$$
	\left\{\begin{array}{rcl}
	(f\cup g)_{|} 
	&=& 1 \in B \\ 
	(f\cup g)_\tau(x_1, \ldots, x_{n}) 
	&=& g_k\Big((g\cup f)_{\tau_1}(x_1, \ldots, x_{j_1-1}) x_{j_1},~\ldots~,~ (g\cup f)_{\tau_k}	(x_{j_{k-1}+1}, 
		\ldots, x_{n-1}) x_n\Big) \in B 
	\end{array}\right.
$$
\end{enumerate}
\end{defn}

\begin{rmk}
\label{rmk:def_of_f_tau}
\begin{enumerate}
\item
	We remark that $(f\cup g)_\tau$ is defined by alternating between $f$ and $g$. Note that $(f\cup g)$ as a whole (without indices) is not defined, in particular, it is not an element of $\Mult[[B]]$. 
\item
	The definition of $f_\tau$ can be extended to a slightly more general setting. If $C$ is a unital algebra, $B$ is a left $C$-module, and $f=(f_n)_{n\ge 0}$ consists of linear maps $f_n:B^{\otimes n} \to C$, then $f_{\tau}$ is still defined, for any $\tau\in Y$. In this case, $f_\tau$ will be a linear map from $B^{\otimes |\tau|}$ to $C$.
\end{enumerate}
\end{rmk}

\begin{ex}
Suppose 
$$
\tau = {\begin{tikzpicture}[scale = 1, baseline={(0,0.5)}]
\draw (0,0) -- (.5,.5) -- (1,1) (.5,1.5) -- (1,2) (-.5,.5) -- (0,1) (1,1) -- (.5,1.5) (0,0) -- (-.5,.5) -- (-1,1)  (0,1) -- (-.5,1.5);
\draw[fill, color=red] (-1,1) circle (.03) node[below]{1};
\draw[fill, color=blue] (-.5,.5) circle (.03) node[below]{2};
\draw[fill, color=red] (-.5,1.5) circle (.03) node[below]{3};
\draw[fill, color=blue] (0,1) circle (.03) node[below]{4};
\draw[fill, color=red] (0,0) circle (.03) node[below]{5};
\draw[fill, color=red] (.5,.5) circle (.03) node[below]{6};
\draw[fill, color=blue] (.5,1.5) circle (.03) node[below]{7};
\draw[fill, color=blue] (1,2) circle (.03) node[below]{8};
\draw[fill, color=red] (1,1) circle (.03) node[below]{9};
\end{tikzpicture}},
$$ 
where only internal vertices are shown. For $f,g \in \Mult[[B]]$, we have
\begin{align*}
	f_\tau(x_1, \ldots, x_9) 
	&= f_3\Big( f_2\big( f_1(x_1)x_2, f_1(x_3)x_4\big) x_5, ~ x_6, ~ f_2(x_7,x_8) x_9\Big)\\
	({\color{blue} f} \cup {\color{red}g})_\tau (x_1, \ldots, x_9) 
	&= {\color{red}g_3\Big( {\color{blue}f_2\big( {\color{red}g_1(x_1)}x_2, {\color{red}g_1(x_3)}x_4\big)} x_5, 
	~ x_6, ~ {\color{blue}f_2(x_7,x_8)} x_9\Big)} \\ 
	({\color{blue}g} \cup {\color{red}f})_\tau (x_1, \ldots, x_9) 
	&= {\color{red}f_3\Big( {\color{blue}g_2\big( {\color{red}f_1(x_1)}x_2, {\color{red}f_1(x_3)}x_4\big)} x_5, 
	~ x_6, ~ {\color{blue}g_2(x_7,x_8)} x_9\Big)}.
\end{align*}
\end{ex}

\begin{nota}
A tree of the form $\tau = \sigma/\tO = \tover \sigma$ will be called a \textit{right-planted tree}.
\end{nota}

The substitution of a tree by right-planted trees behaves well with respect to $f_\tau$:

\begin{lem}
\label{lem:compo_multilinear_series_trees}
Let $\tau \in Y_k$, $\sigma_1 / \tO = \tover{\sigma_1}, \ldots, \sigma_k / \tO = \tover{\sigma_k}\in Y/\tO$. Let $n = \sum_{i=1}^k |\sigma_i / \tO|$. Then for $f \in I \cdot \mathrm{Mult}[[B]]$ and for all $x_1, \ldots, x_n\in B$, 
$$
	f_{\tau\circ\big(\smtover{\sigma_1}, \ldots, \smtover{\sigma_k}\big)}(x_1, \ldots, x_n)
	= f_\tau\Big(f_{\sigma_1}(x_1, \ldots, x_{|\sigma_1|})x_{|\sigma_1|+1}, \ldots, 
	~ f_{\sigma_k}(x_{n-|\sigma_k|}, \ldots, x_{n-1})x_n\Big).
$$
\end{lem}

\begin{proof}
Note that we have used the operadic notation \eqref{eq:operadprod} mentioned in Remark \ref{emk:operad}. We prove the lemma by induction on $|\tau|$. Let us write 
$$
	\tau = \toverrightcomb{\tau_1}{\tau_2}{\tau_k},
$$ 
thus
$$
	f_\tau(x_1,\ldots,x_n)=
	f_k \Big(f_{\tau_1} 
	 (x_1, \ldots, x_{|\tau_1|}) x_{|\tau_1|+1}, \ldots, ~ f_{\tau_k}(x_{n-|\tau_k|}, \ldots, x_{n-1})x_n\Big).
$$
Simplifying notation, we replace the $x_i$'s by $x$. This is justified since the order of the $x_i$'s does not change. We prove the equality of the Lemma by deriving the left-hand-side from the right-hand-side.
\allowdisplaybreaks
\begin{align*}
	f_\tau\Big(&f_{\sigma_1}(x_1, \ldots, x_{|\sigma_1|})x_{|\sigma_1|+1}, \ldots, ~ f_{\sigma_k}(x_{n-|\sigma_k|}, \ldots, x_{n-1})x_n\Big). \\
	&= f_\tau\Big( f_{\sigma_1}(x, \ldots, x)x, \ldots, ~ f_{\sigma_k}(x, \ldots, x)x\Big) \\
	&= f_j\Big(f_{\tau_1}\big(f_{\sigma_1}(x, \ldots, x)x, \ldots, ~ f_{\sigma_{k_1-1}}(x, \ldots, x)x\big)f_{\sigma_{k_1}}(x, \ldots, x)x, \ldots \\
	& \hspace{70pt} \ldots, f_{\tau_j}\big(f_{\sigma_{k_{j-1}+1}}(x, \ldots, x)x, \ldots, ~ f_{\sigma_{k-1}}(x, \ldots, x)x\big)f_{\sigma_k}(x, \ldots, x)x\Big) \\ 
	&\overset{(\star)}{=} f_j\Big( f_{\tau_1\circ \big(\smtover{\sigma_1}, \ldots, \smtover{\sigma_{k_1-1}}\big)}(x, \ldots, x)f_{\sigma_{k_1}}(x, \ldots, x)x, 
		\ldots, f_{\tau_j\circ \big(\smtover{\sigma_{k_{j-1}+1}}, \ldots, \smtover{\sigma_{k-1}}\big)}(x, \ldots, x)f_{\sigma_k}(x, \ldots, x)x \Big) \\ 
	&= f_j\Big( f_{\rho_1}(x, \ldots, x)x, \ldots, f_{\rho_j}(x, \ldots, x)x \Big) \\ 
	&= f_{\smtoverrightcomb{\rho_1}{\rho_j}}(x, \ldots, x)
\end{align*}
with $\rho_i = \big(\tau_i\circ(\smtover{\sigma_{k_{i-1}+1}}, \ldots, \smtover{\sigma_{k_i-1}}) \big) / \sigma_{k_i}$, and therefore
$$
	\toverrightcomb{\rho_1}{\rho_2}{\rho_j} =
	\begin{tikzpicture}[scale = .5, baseline={(0,0.4)}]
	\draw (0,-.5)--(0,0)--(3.5,3.5) (0,0)--(-.5,.5) (1,1)--(.5,1.5) (3,3)--(2.5,3.5); \draw (1.1,2.7) node{\rotatebox[origin=c]{80}{$\ddots$}} (0,0);
	\draw (0,.4) node[above left]{$\sigma_{k_1}$} (1,1.4) node[above left]{$\sigma_{k_2}$} (3,3.4) node[above left]{$\sigma_k$};
	\draw (-1.3, 1.3)--(-2, 2) (-.3, 2.3)--(-1, 3) (1.7, 4.3)--(1, 5);
	\draw (-1.8, 1.4) node[above left]{$\tau_1\circ(\smtover{\sigma_1}, \ldots, \smtover{\sigma_{k_1-1}})$};
	\draw (-.8, 2.7) node[above left]{$\tau_2\circ(\smtover{\sigma_{k_{1}+1}}, \ldots, \smtover{\sigma_{k_2-1}})$};
	\draw (1.2, 4.5) node[above left]{$\tau_j\circ(\smtover{\sigma_{k_{j-1}+1}}, \ldots, \smtover{\sigma_{k-1}})$};
	\end{tikzpicture}
	= \tau \circ\Big(\tover{\sigma_1}, \ldots, \tover{\sigma_k}\Big).
$$
\end{proof}

Note that the induction hypothesis is used on the $\tau_i$ in equality ($\star$), as we use for example
\[
	f_{\tau_1}\big(f_{\sigma_1}(x, \ldots, x)x, \ldots, ~ f_{\sigma_{k_1-1}}(x, \ldots, x)x\big) = f_{\tau_1\circ \big(\smtover{\sigma_1}, \ldots, \smtover{\sigma_{k_1-1}}\big)}(x, \ldots, x)
\]

The use of \eqref{eq:operadprod} in this lemma suggests an operadic point of view. In Appendix \ref{app:operad_pt_view}, a more algebraic (i.e., transparent) proof of Lemma \ref{lem:compo_multilinear_series_trees} is given.


\subsection{The boxed convolution operations}
\label{ssec:boxed_conv}


%
%
%

\begin{defn}
\label{def:R_and_Y^b}
Define a map $R : Y \to Y$ recursively by
\begin{align*}
	R(|) &= | \\ 
	R\Big( \tovunder \sigma\tau \Big) &= \trecb{R(\sigma)}{R(\tau)}.
\end{align*}
Denote  
$$
	Y^{b}_{2n} := R(Y_n) \ \text{ and }\ Y^b := R(Y).
$$
Observe that $Y^b_{2n}\subseteq Y_{2n}$, and that $R$ is injective. Moreover, note that
\begin{equation}
	R\Big(\toverrightcomb{\sigma_1}{\sigma_2}{\sigma_k}\Big) = \mathrm{RC}(R(\sigma_1), \ldots, R(\sigma_k)), \label{eq:R_right_comb_RC}
\end{equation}
where $\mathrm{RC}$ is the rotated comb tree defined by
\begin{equation}
\label{eq:right_comb_RC_def}
	\mathrm{RC}(\sigma_1, \ldots, \sigma_j) := ~\mathbin{
	\begin{tikzpicture}[scale=.25, baseline={(0,.5)}]
	\draw (0,-.5)--(0,0)--(4.5,4.5) (0,0)--(-1.5,1.5) (-1,1)--(-.5,1.5) (4,4)--(2.5,5.5) (3,5)--(3.5,5.5);
	\draw (0,2) node{$\sigma_1$} (4,6) node{$\sigma_j$} (1.3,3.7) node{\rotatebox[origin=c]{80}{$\ddots$}};
	\end{tikzpicture}}.
\end{equation}
\end{defn}

\begin{rmk}
\label{rmk:Yb_Catalan_pair}
The set $Y^b_{2n}$, together with the map $(\sigma, \tau) \mapsto \trecb\sigma\tau$, form a Catalan pair.
\end{rmk}

\begin{prop}
\label{prop:Krewerstree}
It holds that
\begin{align*}
	\varphi(Y^b_{2n}) &= \{P \cup K(P) \mid P \in \mathrm{NCP}_n\},
\end{align*}
where $K(P)$ is the Kreweras complement of the noncrossing partition $P  \in  \mathrm{NCP}$.
\end{prop}

\begin{proof}
Define 
$$
	\NCP^b_{2n} := \{P \cup K(P) \mid P \in \NCP_n\}.
$$
Since $|\NCP^b_{2n}| = |\NCP_n| = \frac{1}{n+1}\binom{2n}{n} = |Y^b_{2n}|$ (the last equality follows from the fact that $Y^b_{2n}$ admits a Catalan map by Remark \ref{rmk:Yb_Catalan_pair}), we only need to show that $\varphi(Y^b_{2n}) \subseteq \NCP^b_{2n}$. Let us show this by induction on $n$. Let $\sigma \in Y^b_{2k}$, $\tau \in Y^b_{2l}$, and write $\varphi(\sigma) = P\cup K(P)$ and $\varphi(\tau) = Q\cup K(Q)$ with $P\in \NCP_k$ and $Q\in \NCP_l$. Then 
\allowdisplaybreaks
\begin{align*}
	\varphi\Big(\trecb\sigma\tau\Big) 
	&= \big(|\underline * \varphi(\sigma)\big) * \big(|\underline * \varphi(\tau)\big) \\
	&= \big(|\underline * (P\cup K(P))\big) * \big(|\underline * (Q\cup K(Q))\big) \\
	&= \big((| \underline * K(P)) * Q\big) \cup \big((P * |) \underline * K(Q)\big) \\ 
	&= \big((| \underline * K(P)) * Q\big) \cup K\big((| \underline * K(P)) * Q\big) \in \NCP^b_{2n}.
\end{align*}
The last two equalities can be seen using the diagrammatical representation of noncrossing partitions. Here is how to represent the second to last-equality:
$$
	
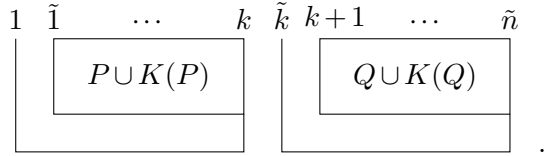
\begin{tikzpicture}
	\draw (0, 1)--(0,-.5)--(3,-.5)--(3,0) (3.5,1)--(3.5,-.5)--(6.5,-.5)--(6.5,0);
	\draw (.5, 0)rectangle(3,1) (4,0)rectangle(6.5,1);
	\draw (1.75, .5)node{$P\cup K(P)$} (5.25, .5)node{$Q\cup K(Q)$};
	\draw (0,1)node[above]{$1$} (.5,1)node[above]{$\tilde 1$} (1.75,1)node[above]{$\cdots$} (3,1)node[above]{$k$} (3.5,1)node[above]{$\tilde k$} (4.2,1)node[above]{$k+1$} (5.4,1)node[above]{$\cdots$} (6.5,1)node[above]{$\tilde n$};
	\end{tikzpicture}~.
$$
The partition on $\{1, \ldots, n\}$ is $(|\underline * K(P)) * Q$, and the partition on $\{\tilde 1, \ldots, \tilde n\}$ is $P * (| \underline * K(Q)) = (P * |) \underline * K(Q)$. For the last equality, one has 
$$
	K\big((| \underline * K(P)) * Q\big) = K\big((| \underline * K(P))\big) \underline * K(Q),
$$ 
thanks to Lemma \ref{lem:Kreweras_change_*}, so we only need to see that $K(|\underline * K(P)) = P*|$. We represent here $(|\underline * K(P)) \cup (P*|)$, which is noncrossing:
$$
	
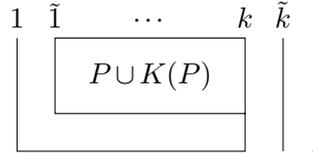
\begin{tikzpicture}
	\draw (0, 1)--(0,-.5)--(3,-.5)--(3,0) (.5, 0)rectangle(3,1) (3.5,1)--(3.5,-.5);
	\draw (1.75, .5)node{$P\cup K(P)$};
	\draw (0,1)node[above]{$1$} (.5,1)node[above]{$\tilde 1$} (1.75,1)node[above]{$\cdots$} (3,1)node[above]{$k$} (3.5,1)node[above]{$\tilde k$};
	\end{tikzpicture}~.
$$
We see that any larger partition than $P*|$ on the right side of $\cup$ would create a crossing, meaning that $K(|\underline * K(P)) = P*|$, as claimed.
\end{proof}

\begin{rmk}
\label{rmk:not}
We note, however, that it is \textit{not} true that for $\tau\in Y$, $\varphi(R(\tau)) = \varphi(\tau) \cup K(\varphi(\tau))$. 
The following example shows that with the help of $\tau = \tA$.
\end{rmk}

\begin{ex}
We exemplify the claim made in Remark \ref{rmk:not} by showing that $\varphi(R(\tA))= \pB \cup K(\pB)$, even though $\varphi(\tA) = | \ |$. Indeed, observe that
$$
	R(\tO)
	=R(| / \tO \backslash |) 
	=(| / \tO \backslash R(|)) / \tO \backslash R(|)
	= \trecb{R(|)}{R(|)}=\tA
$$
and 
$$
	\varphi(\tA) = | \ |.	 
$$
\begin{align*}
	R(\tA)
	=R(\tO / \tO \backslash |) 
	=(| / \tO \backslash R(\tO)) / \tO \backslash R(|)
	= \trecb{R(\tO)}{R(|)}
\end{align*}
and 
$$
	\varphi(\trecb{R(\tO)}{R(|)}) 
	= \varphi(\tO \backslash R(\tO)) * (| \ \underline{*}\ \varphi(|))
	=\big(\emptyset  * (|\ \underline{*}\  \varphi(R(\tO)))\big) * \big(| \ \underline{*}\ \emptyset\big)
	=(|\  \underline{*}\  | \ | ) * \pO 	 
	=\pH,
$$
which is indeed the partition $\pB \cup K(\pB)$, as $K(\pB)=| \ |$.
\end{ex}

The following definition proposes several operator-valued boxed-convolution operations, which are meant to generalize both the classical boxed-convolution and the reduced boxed-convolution defined in the scalar-valued case. See Section 18 of the standard reference \cite{nica_speicher_book} by A.~Nica and R.~Speicher. The motivation for the definition of these convolution products on $\Mult [[B]]$ is given in Section $\ref{sec:meaning_free_proba}$ below.

\begin{defn}
\label{def:modboxedconv}
Let $f,g\in \Mult[[B]]$. Define, for $n\ge 1$, 
\allowdisplaybreaks
\begin{align}
	(f\boxconv g)_n(x_1, \ldots, x_n) 
	&:= \sum_{\tau\in Y_n} (f\cup g)_{R(\tau)}(x_1, 1, x_2, 1, \ldots, 1, x_n, 1) \label{convol1}\\ 
	(f\boxconvline g)_n(x_1, \ldots, x_n) 
	&:= \sum_{\tau\in Y_n} (f\cup g)_{R(\tau)}(1, x_1, 1, x_2, \ldots, 1, x_n) \label{convol2}\\ 
	(f\boxconvred g)_n(x_1, \ldots, x_n) 
	&:= \sum_{\tau\in Y_ {n-1}} (g\cup f)_{\smtunder{R(\tau)}}(x_1, 1, x_2, 1, \ldots, 1, x_n) \label{convol3}\\ 
	(f\boxconvredred g)_n(x_1, \ldots, x_n) 
	&:= \sum_{\tau\in Y_n} (f\cup g)_{\smtunder{R(\tau)}}(1, x_1, 1, x_2, 1, \ldots, 1, x_n, 1) . \label{convol4}
\end{align}
By requiring additionally $(f\boxconv g)_0 = g_0,~ (f\boxconvline g)_0 = g_0,~ (f\boxconvred g)_0 = 0$ and $(f\boxconvredred g)_0 = g_1(1)$, we have defined four series $f\boxconv g,~ f\boxconvline g,~ f\boxconvred g,~ f\boxconvredred g \in \Mult[[B]]$.
\end{defn}

Let us first give some insight into the values of these sums.

\begin{ex}
For any $f,g\in \Mult[[B]]$, the first terms in the series $f\boxconv g$ are

\begin{align*}
	(f\boxconv g)_0 &= g_0, \\
	(f\boxconv g)_1(x_1) &= g_1(f_1(x_1)), \\ 
	(f\boxconv g)_2(x_1, x_2) &= g_1(f_2(x_1, g_1(1)x_2)) + g_2(f_1(x_1),f_1(x_2)).
\end{align*}
The first terms in the series $f\boxconvline g$ are
\begin{align*}
	(f\boxconvline g)_0 &= g_0, \\
	(f\boxconvline g)_1(x_1) &= g_1(f_1(1)x_1), \\ 
	(f\boxconvline g)_2(x_1, x_2) &= g_1(f_2(1, g_1(x_1))x_2) + g_2(f_1(1)x_1,f_1(1)x_2).
\end{align*}
The first terms in the series $f\boxconvred g$ are
\begin{align*}
	(f\boxconvred g)_0 &= 0, \\
	(f\boxconvred g)_1(x_1) &= f_1(x_1), \\ 
	(f\boxconvred g)_2(x_1, x_2) &= f_2(x_1, g_1(1)x_2).
\end{align*}
The first terms in the series $f\boxconvredred g$ are
\begin{align*}
	(f\boxconvredred g)_0 &= g_1(1), \\
	(f\boxconvredred g)_1(x_1) &= g_2(1, f_1(x_1)), \\ 
	(f\boxconvredred g)_2(x_1, x_2) &= g_2(1, f_2(x_1, g_1(1)x_2)) + g_3(1, f_1(x_1), f_1(x_2)).
\end{align*}
\end{ex}

\begin{rmk}
\begin{enumerate}
	\item Note that equation \eqref{convol4} still holds for $n=0$.
	\item If $f,g\in G_B^\dif$, then $f\boxconv g\in G_B^\dif$ and $f\boxconvred g \in G_B^\dif$.
	\item If moreover $f,g\in G_B^I$, then also $f_1(1)\in B^\times$ and $g_1(1)\in B^\times$, so $(f\boxconvline g)_1 \in \GL(B)$ and $(f\boxconvredred g)_0 \in B^\times$, and therefore $f\boxconvline g\in G_B^\dif$ and $f\boxconvredred g \in G_B^\inv$.
\end{enumerate}
\end{rmk}

\begin{lem}
\label{lem:2_boxconv_in_G^I}
Let $f,g\in G_B^I$. Then $f\boxconv g\in G_B^I$ and $f\boxconvred g \in G_B^I$.
\end{lem}
\begin{proof}
We already saw that $f\boxconv g\in G_B^\dif$ and $f\boxconvred g\in G_B^\dif$. Moreover, from the definition of $(f\cup g)_\tau$ we can show by induction on $n\ge 1$ that for all $\tau \in Y_n$, for all $x_1, \ldots, x_n \in B$, we have $(f\cup g)_{\tau}(x_1, \ldots, x_n) = x_1(f\cup g)_{\tau}(1, x_2, \ldots, x_n)$. Therefore
\begin{align*}
	(f\boxconv g)_n(x_1, \ldots x_n) &= \sum_{\tau\in Y_n} (f\cup g)_{R(\tau)}(x_1, 1, x_2, 1, \ldots, 1, x_n, 1)\\ 
	&= \sum_{\tau\in Y_n} x_1(f\cup g)_{R(\tau)}(1, 1, x_2, 1, \ldots, 1, x_n, 1) \\ 
	&= x_1 (f\boxconv g)_n(1, x_2, \ldots x_n).
\end{align*}
Thus $f\boxconv g\in G_B^I$ thanks to Equation \eqref{eq:GBI_alternative_def}, and similarly $f\boxconvred g \in G_B^I$.
\end{proof}

\smallskip

\begin{lem}
\label{lem:4_eq_op_valued_boxconv_compo}
For $f,g\in\mathrm{Mult}[[B]]$, the following holds
\begin{align}
	(f\boxconv g) 
	&= g \circ (f\boxconvred g) \label{eq_conv_compo}\\ 
	(g\boxconvline f) 
	&= f \circ ((f\boxconvredred g)I) \label{eq_convline_compo}.
\end{align}
\end{lem}

\begin{proof}
At the core of these two equalities is the same following fact: for every tree $\tau \in Y$, there exists unique $j\ge 1$, $\sigma_1, \ldots, \sigma_j \in Y$ such that $\tau = \smtoverrightcomb{\sigma_1}{\sigma_j}$. Now, since $R : Y \to Y^b$ is a bijection, we deduce that: for every tree $R(\tau) \in Y^b$, there exists unique $j\ge 1$, $R(\sigma_1), \ldots, R(\sigma_j) \in Y^b$ such that $R(\tau) = R\Big(\smtoverrightcomb{\sigma_1~}{\sigma_j~}\!\Big) =\text{RC}(R(\sigma_1), \ldots, R(\sigma_j))$. The last equality comes from \eqref{eq:R_right_comb_RC}.

Let us derive \eqref{eq_conv_compo}, the other one being done similarly. Suppose $n\ge 1$. Then 
\allowdisplaybreaks
\begin{align*}
	(f\boxconv g)_n(x_1, \ldots, x_n) 
	&= \sum_{\tau\in Y_n} (f\cup g)_{R(\tau)}(x_1, 1, x_2, 1, \ldots, 1, x_n, 1) \\ 
	&{\mkern-80mu} = \sum_{\substack{n=k_1 + \cdots + k_j \\ k_i\ge 1}} 
	\sum_{\sigma_i\in Y_{k_i-1}} (f\cup g)_{\text{RC}(R(\sigma_1), \ldots, R(\sigma_j))}(x_1, 1, \ldots, 1, x_n, 1) \\ 
	&{\mkern-80mu} = \sum_{\substack{n=k_1 + \cdots + k_j \\ k_i\ge 1}} 
	\sum_{\sigma_i\in Y_{k_i-1}} g_j\Big( (g\cup f)_\smtunder{R(\sigma_1)}(x_1, 1, \ldots, 1, x_{k_1}) \cdot 1, 
	~\ldots, ~(g\cup f)_\smtunder{R(\sigma_j)}(x_{n-k_j+1}, 1, \ldots, 1, x_n)\cdot 1\Big) \\ 
	&{\mkern-80mu} = \sum_{\substack{n=k_1 + \cdots + k_j \\ 
	k_i\ge 1}} g_j\Big( (f\boxconvred g)_{k_1}(x_1, \ldots, x_{k_1}), ~\ldots, ~(f\boxconvred g)_{k_j}(x_{n-k_j+1}, \ldots, x_n)\Big) \\ 
	&{\mkern-80mu}= \Big(g \circ (f\boxconvred g)\Big)_n(x_1, \ldots, x_n).
\end{align*}
Note that in the third equality, each $(g\cup f)_\smtunder{R(\sigma_i)}$ takes an odd number of arguments, since $\smtunder{R(\sigma_i)}$ always have an odd number of vertices. Finally, we check that $(f\boxconv g)_0 = g_0 = \Big(g \circ (f\boxconvred g)\Big)_0$.
\end{proof}

For multilinear series in $I\cdot \Mult[[B]]$, the same boxed convolutions, $f\boxconv g$ and $f\boxconvline g$, have another factorisation:

\begin{lem}
\label{lem:4_eq_op_valued_boxconv_mult}
For $f,g\in I\cdot\mathrm{Mult}[[B]]$,
\begin{align}
	(f\boxconv g) 
	&= (f \boxconvred g)\cdot (f \boxconvredred g) \label{eq_conv_mult}\\ 
	(g\boxconvline f) 
	&= (f \boxconvredred g)\cdot(f \boxconvred g) \label{eq_convline_mult}.
\end{align}
Note that it is not necessarily true if we only assume $f,g \in \mathrm{Mult}[[B]]$
\end{lem}

\begin{proof}
As $f,g\in I\cdot \Mult[[B]]$, the following fact is true: If $\sigma, \tau \in Y$,
\begin{equation}
\label{eq:Imult_additionnal_eq}
	(f\cup g)_{\smtovunder\tau\sigma}(x,\ldots, x) = (g\cup f)_\tau(x, \ldots, x)(f\cup g)_{\smtunder\sigma}(x,\ldots, x).
\end{equation}
Indeed, let us write $\tunder\sigma = \toverrightcomb{}{\sigma_2}{\sigma_k}$, so that $\tovunder\tau\sigma = \toverrightcomb{\tau}{\sigma_2}{\sigma_k}$, thus
\allowdisplaybreaks
\begin{align*}
	(f\cup g)_{\smtovunder\tau\sigma}(x,\ldots, x) &= g_k((g\cup f)_{\tau}(x, \ldots, x)x, (g\cup f)_{\sigma_2}(x, \ldots, x)x, 
	\ldots, (g\cup f)_{\sigma_k}(x, \ldots, x)x) \\
	&= (g\cup f)_{\tau}(x, \ldots, x) g_k(x, (g\cup f)_{\sigma_2}(x, \ldots, x)x, \ldots, (g\cup f)_{\sigma_k}(x, \ldots, x)x) \\
	&= (g\cup f)_\tau(x, \ldots, x)(f\cup g)_{\smtunder\sigma}(x,\ldots, x).
\end{align*}
Note that we have used that $g_k(bx, x, \ldots, x) = bx g_k(1,x, \ldots, x) = b g_{k}(x, x, \ldots, x)$ for $b\in B$ in the second equality, which holds since $g\in I\cdot \Mult[[B]].$

Now, the derivation of the two equalities from the statement in the lemma is a similar computation as in the proof of Lemma \ref{lem:4_eq_op_valued_boxconv_compo}. This time, the core fact is that for $R(\tau)\in Y^b$, there exists unique $R(\sigma), R(\rho) \in Y^b$ such that $R(\tau) = \trecb{R(\sigma)}{\hspace{1pt}R(\rho)}$.

Let us derive \eqref{eq_conv_mult}, the other one is done similarly. Suppose that $n\ge 1$. Then
\allowdisplaybreaks
\begin{align*}
	(f\boxconv g)_n(x_1, \ldots, x_n) 
	&= \sum_{\tau\in Y_n} (f\cup g)_{R(\tau)}(x_1, 1, x_2, 1, \ldots, 1, x_n, 1) \\ 
	&= \sum_{\substack{k+l=n-1 \\ k, l\ge 0}} \sum_{\substack{\sigma\in Y_k \\ 
		\rho\in Y_l}} (f\cup g)_{\smtrecb{\underset{~}{R(\sigma)}}{R(\rho)}}(x_1, 1, x_2, 1, \ldots, 1, x_n, 1) \qquad \text{ with } \tau = \tovunder \sigma\rho\\ 
	&\overset{\eqref{eq:Imult_additionnal_eq}}{=} \sum_{\substack{k+l=n-1 \\ k, l\ge 0}} \sum_{\substack{\sigma\in Y_k \\ 
	\rho\in Y_l}} (g\cup f)_{\smtunder{R(\sigma)}}(x_1, 1, x_2, 1, \ldots, 1, x_{k+1}) (f\cup g)_{\smtunder{R(\rho)}}(1, x_{k+2}, 1, \ldots, 1, x_n, 1) \\
	&= \sum_{\substack{k+l=n-1 \\ k, l\ge 0}} (f\boxconvred g)_{k+1}(x_1, x_2, \ldots, x_{k+1}) 
	(f\boxconvredred g)_l(x_{k+2}, \ldots, x_n) \\ 
	&= \Big((f\boxconvred g) (f\boxconvredred g)\Big)_n(x_1, \ldots, x_n).
\end{align*}
Finally, we check that $(f\boxconv g)_0 = 0 = \Big((f \boxconvred g)\cdot (f \boxconvredred g)\Big)_0$, since $(f \boxconvred g)_0 = 0$.
\end{proof}

We consider now the properties of the S-transform (Definition \ref{def:S_transform}) with respect to the first convolution product in Definition \ref{def:modboxedconv}. 

\begin{thm}
\label{thm:S_transf_boxconv}
Let $f,g \in G_B^I$. Then the S-transform satisfies the following identity with respect to the boxed convolution product \eqref{convol1}
$$
	S_{f\boxconv g} = S_g \cdot (S_f \circ (S_g^{-1}IS_g)).
$$
\end{thm}

\begin{proof}
Suppose that $f=IF$ and $g=IG$, for  $F ,G\in G_B^\inv$.
Let us write the following products 
$$
	h := f\boxconv g, 
	\quad 
	h_1 := f\boxconvred g, 	
	\quad 
	\overline h := f\boxconvline g, 
	\quad 
	H_2 := f\boxconvredred g. 
$$
We already noticed that the products $f\boxconv g$ and $f\boxconvred g$ are in $G_B^I$ (Lemma \ref{lem:2_boxconv_in_G^I}). Hence, we can write $h = IH$ and $h_1 = IH_1$. Lemma \ref{lem:4_eq_op_valued_boxconv_compo} then tells us that:
\begin{align}
	h 
	&= g\circ h_1 \label{eq:h_compo} \\
	\overline h 
	&= f \circ (H_2I). \label{eq:hb_compo}
\end{align}
As $f, g\in I\cdot \Mult[[B]]$, $h$ and $\overline h$ also factorize by Lemma \ref{lem:4_eq_op_valued_boxconv_mult}:
\begin{align}
	h 
	&= h_1H_2 \label{eq:h_mult} \\
	\overline h 
	&= H_2 h_1. \label{eq:hb_mult}
\end{align}

First observe that inverting \eqref{eq:h_compo} yields 
$$
	IS_h = (IS_{h_1}) \circ (IS_g) = IS_g \cdot S_{h_1}\circ(IS_g)
$$ 
and thus
$$
	S_h = S_g \cdot S_{h_1}\circ(IS_g). \quad\quad (\star)
$$
Then, the same equality \eqref{eq:h_compo} can be written as 
$$
	IH = (IG)\circ (IH_1) = IH_1 \cdot (G\circ (IH_1)). 
$$
But $IH=IH_1H_2$ by \eqref{eq:h_mult} so
$$
	H_2 = G \circ(IH_1).
$$
Therefore, by \eqref{eq:hb_mult}, $\overline h = H_2IH_1 = (G \circ(IH_1)) \cdot I H_1$ so
$$
	\overline h = (GI)\circ (IH_1).
$$
Now \eqref{eq:hb_compo} can be turned into $(GI)\circ (IH_1) = f \circ \big(G \circ(IH_1) \cdot I\big)$, and then, composing on the left by $IS_f$, 
$$
	(IS_f) \circ(GI)\circ (IH_1) = G \circ(IH_1) \cdot I.
$$
Composing on the right by $(IS_{h_1})$ gives
$$
	(IS_f) \circ(GI) = G I S_{h_1}.
$$
But the left-hand side equals $GI\cdot (S_f \circ(GI))$, so
$$
	S_f \circ(GI) = S_{h_1}.
$$
Going back to $(\star)$, we can now say that
$$
	S_h = S_g \cdot S_f \circ(GI)\circ(IS_g).
$$
And we get the desired equality by Corollary \ref{cor:other_expr_S_f-1IS_f}.

\end{proof}

Note the double use of \eqref{eq:h_compo}, which may indicate that there should be a more direct proof, using only one time each equality.


\section{Operator-valued free probability}
\label{sec:meaning_free_proba}


\subsection{The twisted factorisation of the S-transform}
\label{ssec:twisted_fact_S_transf}

Suppose $B$ is a unital algebra over the field $\mathbb{K}$ of zero characteristic. We recall some definitions about operator-valued free probability. See \cite{MingoSpeicher2017} for more details and background.

\begin{defn}\cite{MingoSpeicher2017}
\label{def:op_valued_free_proba1}
\begin{enumerate}
	\item A \textit{unital $B$-algebra} is a unital $K$-algebra $\mathcal A$, that is a bi-module over $B$, and such that for all $a_1,a_2\in \mathcal A, b_1, b_2\in B$, 
	\begin{align*}
		(b_1 \cdot a_1) \cdot b_2 &= b_1 \cdot (a_1 \cdot b_2)\\
		b_1 \cdot (a_1 \cdot a_2) &= (b_1 \cdot a_1) \cdot a_2 \\
		(a_1 \cdot a_2) \cdot b_1 &= a_1 \cdot (a_2 \cdot b_1) \\
		(a_1\cdot b_1)\cdot a_2 &= a_1 \cdot (b_1 \cdot a_2),
	\end{align*}
	where all products and actions are denoted by $\cdot$. In other words, we ask that any product $x_1x_2x_3 = x_1\cdot x_2\cdot x_3$ is uniquely defined, for $x_1, x_2, x_3\in B \cup \mathcal A$. 

	\item A \textit{$B$-valued probability} space is a pair $(\mathcal A, \bb E)$, where $\mathcal A$ is a unital $B$-algebra, and $\bb E : \mathcal A \to B$ is a unital linear map, $\bb E(1) = 1$, such that for all $x, y\in B$ and for all $a \in \mathcal A$, 
		
	\begin{equation}
	\label{eq:Bbimod}
		\bb E(xay) = x \bb E(a) y.
	\end{equation}
	We say that $\bb E$ is $B$-balanced.

	Elements $a \in \mathcal A$ are called \textit{(non-commutative) random variables}. 
	Their distribution is usually given by their \textit{moment series} $M^a\in \Mult[[B]]$ defined by
	$$
		M^a(x_1, \ldots, x_n) = \bb E(a x_1 a x_2 a\cdots a x_n a).
	$$
	\item (Freeness)
	A collection of $B$-subalgebras $(\mathcal A_i)_{i\in I}$ of $\mathcal A$ is said to be \textit{free} if $\bb E(a_1 \cdots a_n) = 0$ when the following conditions holds:
	\begin{itemize}
		\item $n\ge 1$,
		\item $\bb E(a_i) = 0$ for $1\le i\le n$,
		\item $a_i \in \mathcal A_{j_i}$ for $1\le i\le n$,
		\item $j_i \neq j_{i+1}$ for $1\le i< n$.
	\end{itemize}
	A collection of elements $(a_i)_{i\in I} \subseteq \mathcal A$ is said to be free if their respective generated $B$-subalgebras are free. In this paper, for simplicity we shall focus on the case of two free elements $a,b \in \mathcal A$.
\end{enumerate}
\end{defn}

\begin{rmk}
	In the definition of a unital $B$-algebra, note that the first rule among the four is included in the notion of $B$-bi-module axioms, but not the three others. We emphasize that $B$ is not assumed to be a subset of $\mathcal A$. However, if $B$ is a sub-algebra of $\mathcal A$, then $\mathcal A$ is automatically a $B$-algebra.
\end{rmk}

Next we define operator-valued cumulants in terms of operator-valued moment-cumulant relations using tree-indexed series of multilinear functions (Definition \ref{def:treenesting}). 

\begin{defn}
\label{def:op_valued_free_proba2}
\begin{enumerate}
\item
	The series of operator-valued cumulants $\kappa=(\kappa_n)_{n\ge 0}$ with $\kappa_n:\mathcal A^{\otimes n}\to B$ is defined by requiring that for all $n\ge 0$, $a_1, \ldots, a_n\in \mathcal A$,
	\begin{equation}
	\label{eq:def_of_cumulants}
		\bb E(a_1 \cdots a_n) = \sum_{\tau\in Y_n} \kappa_\tau (a_1, \ldots, a_n).
	\end{equation}
	
\item		
	If $a\in \mathcal A$, define its corresponding cumulant series $K^a\in \Mult[[B]]$ by
	$$
		K^a_n(x_1, \ldots, x_n) := \kappa(a, x_1 a , \ldots, x_n a).
	$$
	We also define $k^a := I K^a \in G_B^I$:
	$$
		k^a_n(x_1, \ldots, x_n) := x_1\kappa(a, x_2 a , \ldots, x_n a)
							= \kappa(x_1 a, x_2 a , \ldots, x_n a).
	$$
	Finally, when $\bb E(a) \in B^\times$, we define the S-transform of the operator-valued random variable $a$ by
	$$
		S_a = S_{k^a},
	$$
	which is well-defined since $\kappa(a) = \bb E(a) \in B^\times$ implying $K^a\in G_B^\inv$ and $k^a\in G_B^I$.
\end{enumerate}
\end{defn}

\begin{rmk}
\label{rmk:cumulants}
\begin{enumerate}
\item
	Note that in the definition of cumulants \eqref{eq:def_of_cumulants}, the notation of $\kappa_\tau$ is based on Remark \ref{rmk:def_of_f_tau} and Definition \ref{def:treenesting}. As $\kappa_{\smtrightcomb n} = \kappa_n$ by definition, equality \eqref{eq:def_of_cumulants} can be rewritten as
	\[
		\kappa_n(a_1, \ldots, a_n) = \bb E(a_1 \cdots a_n) - \sum_{\substack{\tau\in Y_n \\ \tau \neq \smtrightcomb n}} \kappa_\tau(a_1, \ldots, a_n).
	\]
	Since $\kappa_\tau$ is defined from the maps $\kappa_k, k<n$ for $\tau\neq \trightcomb n$, this shows inductively that $\kappa$ is well-defined.
\item
	Note that $\kappa$ inherits the $B$-balancedness \eqref{eq:Bbimod} from $\bb E$. More precisely, for $a_1, \ldots, a_n\in \mathcal A$ and $x_0, \ldots, x_n \in B$, we have
	\begin{equation}
		\label{eq:kappa_balanced}
		\kappa(x_0a_1x_1, a_2x_2, \ldots, a_nx_n) = x_0\kappa(a_1, x_1a_2, \ldots, x_{n-1}a_n) x_n.
	\end{equation}
\end{enumerate}
\end{rmk}

\medskip

In the next sections, we fix a $B$-valued probability space $(\mathcal A, \bb E)$. The following theorem is the operator-valued version of the well-known statement that "freeness $\iff$ mixed cumulants vanish".

\begin{thm}{\rm{\cite[Thm.~9.4]{speicher_notes_op_valued}}}
\label{thm:freeness_mixed_cum_vanish}
Elements $a,b\in \mathcal A$ are free if and only if for all $n\ge 1$, $c_1, \ldots, c_n \in \{a, b\}$ such that the $c_i$ are not all equal, and for all $x_1, \ldots, x_{n-1} \in B$, we have 
$$
	\kappa_n (c_1, x_1c_2,  \ldots, x_{n-2}c_{n-1}, x_{n-1}c_n) = 0.
$$
\end{thm}

We are now interested in products $ab\in \mathcal A$, with fixed free elements $a,b\in \mathcal A$. Let us first give a formula for the cumulants of the $\mathcal A$-product $ab$ as a function of the cumulants of $a$ and of $b$. This formula motivates the definition of one of the boxed convolution operation given in Definition \ref{def:modboxedconv}.

\begin{prop}
\label{prop:formula_cum_product_free_vars}
If $a, b\in \mathcal A$ are free, then $k^{ab} = k^a \boxconv k^b$. Explicitly, 
$$
	k^{ab}_n(x_1, \ldots, x_n) = \sum_{\tau\in Y_n} (k^a\cup k^b)_{R(\tau)}(x_1, 1, x_2, 1, \ldots, 1, x_n, 1).
$$
\end{prop}

\begin{proof}
We devote Section \ref{ssec:formula_cum_prod_free_vars_trees} to the proof of this statement. The formula is also stated in \cite{T-transf}, using a different formalism.
\end{proof}

Theorem \ref{thm:S_transf_boxconv} then implies immediately the twisted factorisation of the S-transform:

\begin{thm}
\label{thm:S_transf_mult_free_proba}
If $a,b\in \mathcal A$ are free and $\bb E(a)$ and $\bb E(b)$ are invertible in $B$, then
\begin{equation}
\label{eq:twiestedmultipl}
	S_{ab} = S_b \cdot S_a \circ (S_b^{-1}IS_b).
\end{equation}
\end{thm}


\subsection{The U-transform and special cases of the twisted factorisation}
\label{ssec:special_cases_U_transf}

Since in the scalar-valued case, the extra composition by $S_b^{-1}IS_b$ does not appear in \eqref{eq:twiestedmultipl}, a natural question to ask is when this term vanishes. First note that it has some properties itself.

\begin{defn}
For $f=IF\in G_B^I$, define its U-transform to be 
\begin{equation}
\label{eq:Utransform}
	U_f := S_f^{-1}IS_f.
\end{equation}
For an element $a\in \mathcal A$, we define also its U-transform in terms of its S-transform: $U_a := U_{k^a} = S_a^{-1}IS_a$.
\end{defn}

Observe that Theorem \mbox{\ref{thm:S_transf_boxconv}\strut} and Theorem \mbox{\ref{thm:S_transf_mult_free_proba}\strut} can be rewritten as 
$$
	S_{f\boxconv g} = S_g\cdot S_f \circ U_g
$$ 
respectively  
$$
	S_{ab} = S_b \cdot S_a \circ U_b.
$$

\noindent It is interesting to note that the U-transform of an element in $\mathcal A$ can be defined directly in terms of its moment series.

\begin{lem}
\label{lem:other_expr_U_transf}
For $a\in \mathcal A$, its U-transform has the following expressions:
\begin{align*}
	U_a = S_a^{-1}IS_a 
		&= (K^aI)\circ (IK^a)^{\circ-1} \\
		&= (M^aI)\circ (IM^a)^{\circ-1}.
\end{align*}
\end{lem}

\begin{proof}
The equality $S_a^{-1}IS_a = (K^aI)\circ (IK^a)^{\circ-1}$ follows immediately from Corollary \ref{cor:other_expr_S_f-1IS_f}. It remains to show that $(K^aI)\circ (IK^a)^{\circ-1} = (M^aI)\circ (IM^a)^{\circ-1}$. According to \cite{speicher_short_proof} (see also \cite{dykema2007strans}), we have the following equalities:
\begin{align*}
	M^a 
	&= K^a\circ (I + I M^a I) \cdot (1+IM^a) \\
	&= (1+M^aI) \cdot K^a\circ (I + I M^a I) 
\end{align*}
which rewrite as 
\begin{align*}
	M^a I 
	&= (K^a I)\circ (I + I M^a I) \\
	I M^a 
	&= (IK^a)\circ (I + I M^a I) 
\end{align*}
hence the result.
\end{proof}

We can now provide an answer to the question of when the factorisation simplifies:

\begin{prop}
Suppose $a,b\in \mathcal A$ are free and both $\bb E(a)$ and $\bb E(b)$ are invertible in $B$. If $K^a$ is constant or $M^b I = I M^b$, then $S_{ab} = S_b \cdot S_a$.
\end{prop}

\begin{proof}
If $K^a$ is constant, so is $S_a$. Therefore $S_a \circ U_b = S_a$. If $M^b I = I M^b$, then by Lemma \ref{lem:other_expr_U_transf}, $U_b = (I M^b)\circ (M^b I)^{\circ-1} = I$, so $S_a\circ U_b = S_a$. In both cases, according to Theorem \ref{thm:S_transf_mult_free_proba}, $S_{ab} = S_b\cdot S_a \circ U_b = S_b\cdot S_a$.
\end{proof}

\begin{rmk}
We have not been able to identify a counterexample for the following reverse implication: if $S_{ab} = S_b \cdot S_a$, then either $K^a$ is constant or $M^b I = I M^b$. It amounts to finding a solution for $f = f\circ g$ with $f\in \Mult[[B]]$ non-constant and $g \in G_B^\dif - \{I\}$, where $f$ plays the role of $S_a$, and $g$ the role of $(M^bI)\circ (IM^b)^{\circ-1}$.
\end{rmk}

We note that the U-transform has additional structure with respect to composition and the boxed convolution product defined in \eqref{convol1}:

\begin{lem}
Suppose $f,g \in G_B^I$. Then
$$
	U_{f\boxconv g} = U_f \circ U_g.
$$
Suppose $a,b\in \mathcal A$ free and both $\bb E(a)$ and $\bb E(b)$ are invertible in $B$. Then $U_{ab} = U_a \circ U_b.$
\end{lem}

\begin{proof}
A simple computations shows that
\begin{align*}
	U_{f\boxconv g} &= S_{f\boxconv g}^{-1} \cdot I \cdot S_{f\boxconv g} \\ 
	&= (S_g \cdot S_f \circ U_g)^{-1} \cdot I \cdot S_g \cdot S_f \circ U_g \\ 
	&= S_f^{-1} \circ U_g\cdot S_g^{-1} \cdot I \cdot S_g \cdot S_f \circ U_g \\ 
	&= S_f^{-1} \circ U_g\cdot U_g \cdot S_f \circ U_g \\ 
	&= (S_f^{-1} \cdot I \cdot S_f) \circ U_g \\
	&= U_f \circ U_g.
\end{align*}
The second statement is a simple consequence of the first since $k^{ab} = k^a\boxconv k^b$. 
\end{proof}

\begin{rmk}
Note, however, that the relation $U_{ab} = U_a \circ U_b$ is not very informative. For example, if the algebra $B$ is commutative, then $U_a=I$ for all $a\in \mathcal A$.
\end{rmk}

One can define a left-version of the S-transform.

\begin{defn}
Let $f=IF\in G_B^I$. We define $S'_F$
$$
	S'_FI = (FI)^{\circ-1}.
$$
\end{defn}

\begin{lem}
We have the following relations
\begin{align*}
	S_F 
	&= S'_F \circ U_F \\ 
	U_F^{\circ-1} 
	&= S'_FIS_F^{\prime-1}.
\end{align*}
\end{lem}

\begin{lem}
Suppose $a,b\in \mathcal A$ free, $\bb E(a)$ and $\bb E(b)$ invertible in $B$. Then
\begin{align*}
	S_{ab} 
	&= S_b \cdot S_a\circ U_b = (S'_b \cdot S_a) \circ U_b \\ 
	S'_{ab} 
	&= S'_b \circ U_a^{\circ-1}\cdot S'_a = (S'_b \cdot S_a) \circ U_a^{\circ-1}.
\end{align*}
\end{lem}

\noindent Both lemmas can be verified by simple algebraic manipulations.


\subsection{A formula for the cumulants of a product of free variables}
\label{ssec:formula_cum_prod_free_vars_trees}

The goal of this section is to prove Proposition \ref{prop:formula_cum_product_free_vars}, which states that for $a,b \in \mathcal A$ free, the cumulant $k^{ab}$ factorizes as $k^{ab} = k^a \boxconv k^b$. In other words, an expression of $k^{ab}$ is given, which depends on the cumulants $k^a$ and $k^b$. First, by definition of cumulants, we have
\begin{align}
	\bb E(x_1 a b x_2 \cdots x_n a b) &= \sum_{\tau\in Y_n} k^{ab}_\tau(x_1, \ldots, x_n). \label{eq:def_k^ab}
\end{align}

Here we should emphasize the use of planar binary trees on the righthand side. 

A first step in the proof is to derive an expression of $\bb E(x_1 a b x_2 \cdots x_n a b)$ depending on $k^a$ and $k^b$. One can write, by definition of cumulants,
$$
	\bb E(x_1 a b x_2 \cdots x_n a b) = \sum_{\tau\in Y_{2n}} \kappa_\tau(x_1a, b, x_2a, b, \ldots, x_n a, b).
$$
We will see that some of the trees in the sum actually have no contribution (Lemma \ref{lem:not_split_zero}), and that the other ones can be expressed with $k^a$ and $k^b$ (Lemma \ref{lem:kappa_Ybp_equals_cup}). We will obtain the desired equality in Lemma \ref{lem:other_expr_Exab}. We first need to define some more notation on binary trees.

\begin{defn}
Define recursively the sets of trees
\begin{align*}
	Y^{be}_0 &:= \{\ | \ \}, \\ 
	Y^{bo}_{2n+1} &:= \{ \tovunder{\tau_1}{\tau_2}, \tau_i \in Y^{be}_{2k_i}, k_1+k_2 = n\} \qquad \text{ for }n\ge 0, \\ 
	Y^{be}_{2n} &:= \{ \tovunder{\tau_1}{\tau_2}, \tau_1 \in Y^{bo}_{2k_1+1}, \tau_2 \in Y^{be}_{2k_i}, k_1+k_2 = n-1\} \qquad \text{ for }n\ge 1. \\ 
\end{align*}
Denote also $Y^{be} := \bigcup_{n\ge 0} Y^{be}_{2n}$ and $Y^{bo} := \bigcup_{n\ge 0} Y^{bo}_{2n+1}$. More concisely,
\begin{align}
	Y^{be} &= \tovunder{Y^{bo}}{\, Y^{be}} \cup \{|\}, 	\label{eq:rec_def_Ybe}\\
	Y^{bo} &= \tovunder{Y^{be}}{\, Y^{be}}. 			\label{eq:rec_def_Ybo}
\end{align}
\end{defn}

The grading of elements in $Y^{be}$ and $Y^{bo}$ is the usual one, i.e., by number of vertices. Indeed, if $\tau\in Y^{be}_{2n}$ and $\tilde\tau \in Y^{bo}_{2n+1}$, then $|\tau| = 2n$ respectively $|\tilde\tau| = 2n+1$. Note that for positive $n$, 
$$
	Y^{be}_{2n} = \Big\{\ \tovovunder{\tau_1}{\tau_2}{\tau_3} \ , \tau_i \in Y^{be}_{2k_i}, k_1+k_2+k_3 = n-1 \Big\}.
$$

\begin{defn}
\label{def:treesplitting}
Let $\tau \in Y$ and number its vertices according to the left-to-right order. We say that the tree $\tau$ \textit{splits} if all right arms of $\tau$ have elements of the same parity, i.e., either all odd or all even.
\end{defn}

\begin{rmk}
\label{rmk:def_Ybe_Ybo}
\begin{enumerate}
	\item Going back to Definition \ref{def:R_and_Y^b}, we see that $Y^b_{2n} \subseteq Y^{be}_{2n}$.
	\item The set $Y^{be}$ is closed under the over and under operations, $/$ and $\backslash$. 
	Therefore if $\sigma\in Y_n$, $\tau_1, \ldots, \tau_n\in Y^{be}$, then $\sigma\circ(\tau_1, \ldots, \tau_n)\in Y^{be}$.
	\item It is also true that 
	\begin{align*}
		Y^{b{\color{blue}e}} 
		&= \toverrightcomb{\scriptstyle Y^{b{\color{red}o}}}{\scriptstyle Y^{b{\color{red}o}}}{\scriptstyle Y^{b{\color{red}o}}} &
		Y^{b{\color{red}o}} 
		&= \toverrightcomb{\scriptstyle Y^{b{\color{blue}e}}}{\scriptstyle Y^{b{\color{red}o}}}{\scriptstyle Y^{b{\color{red}o}}}
	\end{align*}
	with the possibility of having any number $n \ge 0$ of appearances of $Y^{b{\color{red}o}}$, and only one appearance of $Y^{be}$ in the diagram of $Y^{bo}$.
\end{enumerate}
\end{rmk}

\begin{lem}
\label{lem:splits_iff_Ybplus}
Let $\tau\in Y_n$, $n\ge 0$. Then $\tau$ splits (Definition \ref{def:treesplitting}) if and only if $\tau \in Y^{be}\cup Y^{bo}$. In particular, 
$$
	\varphi(Y^{be}_{2n}) = \Big\{ P\in \mathrm{NCP}_{2n} \mid P = Q_1 \cup Q_2, Q_1, Q_2 \in \mathrm{NCP}_n \Big\}.
$$
\end{lem}
\begin{proof}
Let us proceed by induction on the size of $\tau$. It is obviously true for $|\tau|=0$. Then, let us write 
$$
	\tau = \tovunder{\sigma}{\rho}.
$$ 
Number the vertices of $\tau$ according to the left-to-right order, and denote by $r$ the root's number. The root is in the same right arm as the last vertex, numbered $|\tau|$. So if $\tau$ splits, then $r$ is of the same parity as $|\tau|$, so $|\rho|=|\tau|-r$ is even, and $|\sigma| = |\tau|-|\rho|-1$ is of different parity than $|\tau|$.  Moreover, if $\tau$ splits, $\rho$ and $\sigma$ also split. We now note that these properties characterize the splitting of $\tau$: $\tau$ splits if and only if $|\rho|$ is even, $|\sigma|$ is of different parity than $|\tau|$, and $\sigma$ and $\rho$ split. The result follows from the recursive definitions \eqref{eq:rec_def_Ybe} and \eqref{eq:rec_def_Ybo}.
\end{proof}

Only the trees that split are relevant in our calculation:

\begin{lem}
\label{lem:not_split_zero}
Let $\tau\in Y_{2n}$ be such that it does not split. Then for all $x_1, y_1, \ldots, x_n, y_n \in B$,
$$
	\kappa_\tau(x_1a, y_1b, \ldots, x_na, y_nb) = 0
$$
\end{lem}

\begin{proof}
Let us write
$$
	\tau = \toverrightcomb{\tau_1}{\tau_2}{\tau_m}.
$$ 
If there are even and odd elements in the right arm containing the root, then 
$$
	\kappa_\tau(x_1a, y_1b, \ldots, x_na, y_nb) = \kappa_m(\ldots, \kappa_{\tau_i}(\ldots) a, \ldots, \kappa_{\tau_j}(\ldots) b, \ldots) = 0
$$
for some $1\le i,j\le m$. The expression is zero thanks to Theorem \ref{thm:freeness_mixed_cum_vanish}. Otherwise, one of the $\tau_i$ does not split so $\kappa_{\tau_i}(x_{j_{i-1}+1}a, y_{j_{i-1}+1}b, \ldots, x_{j_i}a) = 0$ by induction. Therefore
$$
	\kappa_\tau(x_1a, y_1b, \ldots, x_na, y_nb) 
	= \kappa_m\Big(\kappa_{\tau_i}(x_1a, y_1b, \ldots, x_{j_1}a) y_{j_1} b, \ldots, 
	\underset{=0}{\underbrace{\kappa_{\tau_i}(x_{j_{i-1}+1}a, y_{j_{i-1}+1}b, \ldots, x_{j_i}a)}} y_{j_i} b, \ldots\Big) = 0.
$$

\end{proof}

We can now write $\kappa_\tau(x_1a, y_1b, \ldots, x_na, y_nb)$ as a function of $k^a$ and $k^b$, when $\tau$ splits. We will only need this for $\tau\in Y^{be}$, but it is nice to see that there is a similar expression for $\tau\in Y^{bo}$, and it also helps the inductive proof.

\begin{lem}
\label{lem:kappa_Ybp_equals_cup}
If $\tau \in Y^{be}_{2n}$, $\sigma \in Y^{bo}_{2n+1}$, then for all $x_1, y_1, \ldots, x_n, y_n, x_{n+1} \in B$,
\begin{align*}
	\kappa_\tau(x_1a, y_1b, \ldots, x_na, y_nb) &= (k^a\cup k^b)_\tau(x_1, y_1, \ldots, x_n, y_n), \\
	\kappa_\sigma(x_1a, y_1b, \ldots, y_nb, x_{n+1}a) &= (k^b\cup k^a)_\sigma(x_1, y_1, \ldots,x_n, y_n, x_{n+1}).
\end{align*}

\end{lem}
\begin{proof}
We work by induction on both statements at the same time, and rely on item 3 of Remark \ref{rmk:def_Ybe_Ybo}. For the first one, we can write write 
$$
	\tau = \toverrightcomb{\tau_1}{\tau_2}{\tau_k}
$$ 
with $\tau_i\in Y^{bo}$. Therefore, by induction
\begin{align*}
	\lefteqn{\kappa_\tau(x_1a, y_1b, \ldots, x_na, y_nb)}\\ 
	&= \kappa_k\Big(\kappa_{\tau_1}(x_1a, y_1b, \ldots, y_{j_1-1}b, x_{j_1}a) y_{j_1}b, \ldots, \kappa_{\tau_k}(x_{j_{k-1}+1}a, y_{j_{k-1}+1}b \ldots, y_{n-1}b, x_na) y_nb\Big) \\ 
	&= k^b\Big(\kappa_{\tau_1}(x_1a, y_1b \ldots, y_{j_1-1}b, x_{j_1}a) y_{j_1}, \ldots, \kappa_{\tau_k}(x_{j_{k-1}+1}a, y_{j_{k-1}+1}b \ldots, y_{n-1}b, x_na) y_n\Big) \\
	&= k^b\Big((k^b\cup k^a)_{\tau_1}(x_1, y_1, \ldots, y_{j_1-1}, x_{j_1}) y_{j_1}, \ldots, (k^b\cup k^a)_{\tau_k}(x_{j_{k-1}+1}, y_{j_{k-1}+1}, \ldots, y_{n-1}, x_n) y_n\Big) \\ 
	&= (k^a\cup k^b)_\tau(x_1, y_1, \ldots, x_n, y_n)
\end{align*}

For the second statement, we can again write 
$$
	\sigma = \toverrightcomb{\tau}{\sigma_1}{\sigma_k}
$$ 
with this time $\tau\in Y^{be}$, $\sigma_i\in Y^{bo}$, and by induction
\begin{align*}
	\kappa_\sigma&(x_1a, y_1b, \ldots, y_nb, x_{n+1}a)\\
	&= \kappa_{k+1}\Big(\kappa_{\tau}(x_1a, y_1b, \ldots, x_{j_1-1}a, y_{j_1-1}b) x_{j_1}a, \kappa_{\sigma_1}(y_{j_1}b, x_{j_1+1}a, \ldots, x_{j_2-1}a, y_{j_2-1}b)x_{j_2}a, \ldots \\ 
	& \hspace{230pt} \ldots, \kappa_{\sigma_k}(y_{j_{k+1}}b, x_{j_{k+1}+1}a, \ldots, x_na, y_nb) x_{n+1}a\Big) \\ 
	&= k^a\Big(\kappa_{\tau}(x_1a, y_1b, \ldots, x_{j_1-1}a, y_{j_1-1}b) x_{j_1}, \kappa_{\sigma_1}(y_{j_1}b, x_{j_1+1}a, \ldots,x_{j_2-1}a, y_{j_2-1}b)x_{j_2}\ldots \\ 
	&\hspace{230pt} \ldots, \kappa_{\sigma_k}(y_{j_{k+1}}b, x_{j_{k+1}+1}a,\ldots, x_na, y_nb) x_{n+1}\Big) \\
	&= k^a\Big((k^a\cup k^b)_{\tau}(x_1, y_1, \ldots, x_{j_1-1}, y_{j_1-1}) x_{j_1}, (k^a\cup k^b)_{\sigma_1}(y_{j_1}, x_{j_1+1}, \ldots, x_{j_2-1}, y_{j_2-1})x_{j_2},\ldots \\ 
	&\hspace{230pt} \ldots, (k^a\cup k^b)_{\sigma_k}(y_{j_{k+1}}, x_{j_{k+1}+1}, \ldots, x_n, y_n) x_{n+1}\Big) \\ 
	&= (k^b\cup k^a)_{\sigma}(x_1, y_1, \ldots, y_n, x_{n+1}).
\end{align*}

Note that we use the second equality inductively, but with the roles of $a$ and $b$ exchanged.
\end{proof}

We can finally get the second expression for $\bb E(x_1 a b x_2 \cdots x_n a b)$ using only $k^a$ and $k^b$:

\begin{lem}
\label{lem:other_expr_Exab}
For $x_1, \ldots, x_n\in B$:
$$
	\bb E(x_1 a b x_2 \cdots x_n a b) = \sum_{\sigma\in Y^{be}_{2n}} (k^a\cup k^b)_\sigma(x_1, 1, x_2, 1, \ldots, x_n, 1).
$$
\end{lem}

\begin{proof}
By definition of cumulants
$$
	\bb E(x_1 a b x_2 \cdots x_n a b) = \sum_{\tau\in Y_{2n}} \kappa_\tau(x_1a, b, x_2a, b, \ldots, x_n a, b).
$$
From Lemmas \ref{lem:splits_iff_Ybplus} and \ref{lem:not_split_zero}, only the $\tau \in Y^{be}_{2n}$ have a non-zero contribution:
$$
	\bb E(x_1 a b x_2 \cdots x_n a b) = \sum_{\tau\in Y^{be}_{2n}} \kappa_\tau(x_1a, b, x_2a, b, \ldots, x_n a, b).
$$
It finally suffices to rewrite $\kappa_\tau(x_1a, b, x_2a, b, \ldots, x_n a, b)$ as $(k^a\cup k^b)_\tau(x_1, 1, x_2, 1, \ldots, x_n, 1)$ thanks to Lemma \ref{lem:kappa_Ybp_equals_cup}.
\end{proof}

Therefore, combining this with equality \eqref{eq:def_k^ab} gives
\begin{align}
	\sum_{\tau\in Y_n} k^{ab}_\tau(x_1, \ldots, x_n) 
	&= \sum_{\sigma\in Y^{be}_{2n}} (k^a\cup k^b)_\sigma(x_1, 1, x_2, 1, \ldots, x_n, 1). \label{eq:sum_kab_ka_kb}
\end{align}

It remains now to extract the expression for $k^{ab}_n$ from this equality, this will take a bit more work. The idea is that each term $k^{ab}_\tau(x_1, \ldots, x_n) $ in the sum on the left-hand-side corresponds to a sum over a subset of $Y^{be}_{2n}$, that will be called $\Pi(\tau)$, and defined in Definition \ref{def:Pi_tau_subset_Y^be}. Let us first give a lemma that extends Lemma \ref{lem:compo_multilinear_series_trees} to series based on two multilinear series $f,g \in \Mult[[B]]$. Once again, the substitution by right-planted trees behaves well with respect to $(f\cup g)_\tau$.


\begin{lem}
\label{lem:compo_multilinear_series_trees_bis}
Let $f, g\in \mathrm{Mult}[[B]]$ and $x_1, y_1, \ldots, x_n, y_n \in B$. Then, with the notation of omitted ordered indices (as in the proof of Lemma \ref{lem:compo_multilinear_series_trees}),
\begin{enumerate}
	\item If $\rho\in Y^{be}_{2k}$, $\tover{\sigma_1}, \ldots, \tover{\sigma_{2k}}\in Y^{be}/\tO$ with $2n = \sum |\tover{\sigma_i}|$, then
	\begin{align*}
		\lefteqn{(f\cup g)_{\rho \circ \big(\smtover{\sigma_1}, \ \ldots, \ \smtover{\sigma_{2k}} \big)}(x_1, y_1,  \ldots, x_n, y_n)}\\
		&= (f\cup g)_\rho\Big((f\cup g)_{\sigma_1}(x, y, \ldots, x, y)x, (g\cup f)_{\sigma_2}(y, x, \ldots, y, x)y, ~\ldots \\
		&\quad	 
		\ldots~, (f\cup g)_{\sigma_{2k-1}}(x, y, \ldots, x, y)x, (g\cup f)_{\sigma_{2k}}(y, x, \ldots, y, x)y \Big).
	\end{align*}
	\item If $\rho\in Y^{bo}_{2k-1}$, $\tover{\sigma_1}, \ldots, \tover{\sigma_{2k-1}}\in Y^{be}/\tO$ with $2n-1 = \sum |\tover{\sigma_i}|$, then
	\begin{align*}
		\lefteqn{(f\cup g)_{\rho \circ \big(\smtover{\sigma_1}, \ \ldots, \ \smtover{\sigma_{2k-1}} \big)}(x_1, y_1,  \ldots, y_{n-1}, x_n)}\\
		&=  (f\cup g)_\rho\Big((g\cup f)_{\sigma_1}(x, y, \ldots, x, y)x, (f\cup g)_{\sigma_2}(y, x, \ldots, y, x)y, ~
		 \ldots ~, (g\cup f)_{\sigma_{2k-1}}(x, y, \ldots, x, y)x\Big).
	\end{align*}
\end{enumerate}
\end{lem}

\begin{proof}
It is similar to the proof of Lemma \ref{lem:compo_multilinear_series_trees} but with two inductions at the same time. We remark that it does not seem to have a nice operadic interpretation.
\end{proof}

\begin{defn}
\label{def:Pi_tau_subset_Y^be}
Define a map $\Pi : Y \to \mathcal P(Y^{be})$ (where $\mathcal P(Y^{be})$ is the power set, i.e., the set of subsets of $Y^{be}$) recursively by
\begin{align*}
	\Pi (~|~) &= \{~|~\} \\ 
	\Pi \Big(\toverrightcomb{\tau_1}{\tau_2}{\tau_k}\ \Big) &= \Big\{ \rho \circ (\tover{\sigma_1}, \ \tO, \ \ldots, \ \tover{\sigma_k}, \tO \ ), \ \rho \in Y^b_{2k}, \sigma_i \in \Pi(\tau_i) \Big\}.
\end{align*}
\end{defn}

\begin{rmk}
\begin{enumerate}
\item
	Note that in particular, for $n\ge 1$, we have
	\[
		\Pi \Big(\trightcomb n\ \Big) = Y^b_{2n}.
	\]
\item
	It is not obvious at all from the definition, but for $\tau \in Y$, we have, with $P = \varphi(\tau)$,
	$$
		\varphi(\Pi(\tau)) = \Big\{ P\cup K_P(Q) \in \NCP_{2n} \mid Q\in \NCP_n, Q\le P \Big\},
	$$
	where $K_P$ is the relative Kreweras complement. See \cite{nica_speicher_book} for more details on the relative Kreweras complement.
\end{enumerate}
\end{rmk}

\begin{lem}
\label{lem:Y^bp_two_decomp}
Every tree $\tau \in Y^{be}$ has two unique decompositions
\begin{align*}
	\tau &= \rho \circ (\tover{\sigma_1}, \ \tO, \ \ldots, \ \tover{\sigma_k}, \ \tO \ ) \\ 
	\tau &= \rho \circ (\ \tO, \tover{\sigma_1}, \ \ldots, \ \tO, \ \tover{\sigma_k} \ )
\end{align*}
with $\rho \in Y^b$ and $\sigma_1, \ldots, \sigma_k \in Y^{be}$. Note that $k$ can be different in the two decompositions.
\end{lem}

\begin{proof}
By induction on the size $|\tau|$ of $\tau$. 
\end{proof}

\begin{lem}
\label{lem:Y^bp_disjoint_union}
For $n\ge 1$, 
$$
	Y^{be}_{2n} = \bigsqcup_{\tau \in Y_n} \Pi(\tau).
$$
\end{lem}

\begin{proof}
It can be seen by an induction on $n$, using the first part of the previous lemma.
\end{proof}

Therefore it makes sense that the sum over $\sigma\in Y^{be}_{2n}$ in \eqref{eq:sum_kab_ka_kb} can be decomposed as a double sum over $\tau\in Y_n$ and $\sigma \in \Pi(\tau)$. It now only remains to show that the sum over $\Pi(\tau)$ correspond to $k^{ab}_\tau$ on the right-hand-side of \eqref{eq:sum_kab_ka_kb}.

\begin{lem}
\label{lem:extract_sum_Y^bp_convol}
Let $f,g,h \in\mathrm{Mult}[[B]]$, assume that for all $n\ge 0$, $x_1, \ldots, x_n \in B$,
\begin{equation}
	\sum_{\tau\in Y_n} h_\tau(x_1, \ldots, x_n) 
	= \sum_{\sigma\in Y^{be}_{2n}} (f\cup g)_\sigma(x_1, 1, x_2, 1, \ldots, x_n, 1). \label{eq:hyp_lem_extract_sum}
\end{equation}
Then for all $n\ge 0$, $\tau\in Y_n$, $x_1, \ldots, x_n\in B$, 
\begin{equation}
	h_\tau(x_1, \ldots, x_n) 
	= \sum_{\sigma \in \Pi(\tau)} (f\cup g)_\sigma(x_1, 1, x_2, 1, \ldots, x_n, 1). \label{eq:lem_extract_sum}
\end{equation}
In particular,
$$
	h_n(x_1, \ldots, x_n) = \sum_{\sigma \in Y_n} (f\cup g)_{R(\sigma)}(x_1, 1, x_2, 1, \ldots, x_n, 1)
$$
\end{lem}

\begin{proof}
The proof is again an induction on $n$. Let us assume that the result is known for all $k<n$, we will first show that relation (\mbox{\ref{eq:lem_extract_sum}\strut}) holds for $\tau \neq \trightcomb n$, and then for $\tau = \trightcomb n$. If $\tau \neq \trightcomb n$, we can write $\tau = \toverrightcomb{\tau_1}{\tau_2}{\tau_k}$ with $k<n$ and $|\tau_i|<n$. Therefore by induction
\begin{align*}
	h_\tau(x_1, \ldots, x_n) 
	&= h_k(h_{\tau_1}(x_1, \ldots, x_{|\tau_1|})x_{|\tau_1|+1}, \ldots, h_{\tau_1}(x_{n-|\tau_k|}, \ldots, x_{n-1})x_n) \nonumber\\ 
	&{\mkern-80mu}= \sum_{\rho \in Y^b_{2k}} (f\cup g)_\rho \Big(h_{\tau_1}(x_1, \ldots, x_{|\tau_1|})x_{|\tau_1|+1}, 1, \ldots, h_{\tau_1}(x_{n-|\tau_k|}, \ldots, x_{n-1})x_n, 1\Big) \nonumber\\ 
	&{\mkern-80mu}= \sum_{\rho \in Y^b_{2k}} (f\cup g)_\rho \Big(\sum_{\sigma_1 \in \Pi(\tau_1)} (f\cup g)_{\sigma_1} (x_1, 1, \ldots, x_{|\tau_1|}, 1)x_{|\tau_1|+1}, 1, \ldots,\\
	&\hspace{5cm} \sum_{\sigma_k \in \Pi(\tau_k)} (f\cup g)_{\sigma_k} (x_{n-|\tau_k|}, 1, \ldots, x_{n-1}, 1) x_n, 1\Big) \nonumber\\ 
	&{\mkern-80mu}= \sum_{\rho \in Y^b_{2k}} \sum_{\sigma_i \in \Pi(\tau_i)} (f\cup g)_\rho \Big((f\cup g)_{\sigma_1} (x_1, 1, \ldots, x_{|\tau_1|}, 1)x_{|\tau_1|+1}, 1, \ldots, (f\cup g)_{\sigma_k} (x_{n-|\tau_k|}, 1, \ldots, x_{n-1}, 1) x_n, 1\Big) \nonumber\\ 
	&{\mkern-80mu}= \sum_{\rho \in Y^b_{2k}} \sum_{\sigma_i \in \Pi(\tau_i)} (f\cup g)_{\rho \circ \big(\smtover{\sigma_1}, \ \smtO, \ \ldots, \ \smtover{\sigma_k}, \ \smtO \ \big)}(x_1, 1 \ldots, x_n, 1) \nonumber\\
	&{\mkern-80mu}= \sum_{\sigma \in \Pi(\tau)} (f\cup g)_\sigma(x_1, 1, x_2, 1, \ldots, x_n, 1),
\end{align*}
where the second to last equality is a special case of Lemma \ref{lem:compo_multilinear_series_trees_bis}, with $y_i=1$ for all $i$ and $\sigma_j=|$ for $j$ even. Hence equality (\mbox{\ref{eq:lem_extract_sum}\strut}). Then, for $\tau = \trightcomb n$, it suffices to subtract 
$$
	\sum_{\tau\in Y_n, \tau \neq \smtrightcomb n} h_\tau(x_1, \ldots, x_n) 
	= \sum_{\tau\in Y_n, \tau \neq \smtrightcomb n} \ \sum_{\sigma \in \Pi(\tau)} (f\cup g)_\sigma(x_1, 1, x_2, 1, \ldots, x_n, 1)
$$
from both sides of the assumption \eqref{eq:hyp_lem_extract_sum}, and the desired equality appears thanks to Lemma \ref{lem:Y^bp_disjoint_union}. 
\end{proof}

Finally, Proposition \ref{prop:formula_cum_product_free_vars} follows directly from this last Lemma applied to equality \ref{eq:sum_kab_ka_kb}.


\appendix


\section{Other bijections arising from Catalan pairs}
\label{app:other_bij}

\allowdisplaybreaks 

\begin{defn}
\begin{enumerate}
	\item Let $\PT$ be the set of planar rooted trees (or ordered rooted trees), not including the empty tree. For $n\ge 0$, let $\PT_n$ be the set of trees $\tau\in \PT$ with $n$ \textit{edges}.
	
	\item Let $\ST\subseteq \PT$ be the subset of \textit{Schröder trees}, containing trees $t \in \PT$ such that each internal vertex has at least two children. For $\tau\in \ST$, we say that $\tau$ is a \textit{left (resp.~right) Schröder tree} if each leftmost (resp.~rightmost) child of each internal vertex is a leaf. Let $\LST$ (resp.~$\RST$) be the set of left (resp.~right) Schröder trees.
	
	\item Let $\NDPF$ be the set of non-decreasing parking functions, i.e., maps $a:[n]\to [n]$ that are non-decreasing and such that $a(i) \le i$.
\end{enumerate}
\end{defn}

\begin{defn}
Let us define Catalan pairs over $Y$, $\NCP$, $\PT$, $\RST$, $\LST$ and $\NDPF$ by defining their Catalan maps:
\allowdisplaybreaks 
\begin{equation*}
\def\arraystretch{1.4}
\begin{array}{lcccl}
 (Y)\hspace{60pt} & (\sigma, \tau) &\mapsto& \tovunder \sigma \tau &\hspace{150pt} \\
 (Y') & (\sigma, \tau) &\mapsto& \tovunder \tau \sigma \\
	 (\NCP_1) & (P, Q) & \mapsto & P * | \underline * Q \\
	 (\NCP_2) & (P, Q) & \mapsto & P \overline * | * Q \\
	 (\NCP_3) & (P, Q) & \mapsto & (| * P)  \overline * Q \\
	 (\NCP_4) & (P, Q) & \mapsto & (| \underline * P) * Q \\
	 (\NCP_5) & (P, Q) & \mapsto & P * (Q \overline * |) \\
	 (\NCP_6) & (P, Q) & \mapsto & P \underline * (Q * |) \\
	 (\NCP_7) & (P, Q) & \mapsto & (| \underline * K(P)) * Q \\
	 (\NCP_8) & (P, Q) & \mapsto & P * | \underline * K(Q) \\
	 (\PT_1) & (\sigma, \tau) & \mapsto & B_+(\sigma, \tau_1, \ldots, \tau_k) &\text{ where }\tau = B_+(\tau_1, \ldots, \tau_k) \\ 
	 (\PT_2) & (\sigma, \tau) & \mapsto & B_+(\sigma_1, \ldots, \sigma_k, \tau) &\text{ where }\sigma = B_+(\sigma_1, \ldots, \sigma_k) \\ 
	 (\RST_1) & (\sigma, \tau) & \mapsto & B_+(\sigma, \tau_1, \ldots, \tau_k, \onet) &\text{ where }\tau = B_+(\tau_1, \ldots, \tau_k, \onet) \\
	 (\RST_2) & (\sigma, \tau) & \mapsto & B_+(\sigma_1, \ldots, \sigma_k, \tau, \onet) &\text{ where }\sigma = B_+(\onet, \sigma_1, \ldots, \sigma_k) \\
	 (\LST_1) & (\sigma, \tau) & \mapsto & B_+(\onet, \sigma_1, \ldots, \sigma_k, \tau) &\text{ where }\sigma = B_+(\onet, \sigma_1, \ldots, \sigma_k) \\
	 (\LST_2) & (\sigma, \tau) & \mapsto & B_+(\onet, \sigma, \tau_1, \ldots, \tau_k) &\text{ where }\tau = B_+(\tau_1, \ldots, \tau_k, \onet) \\
	 (\NDPF) & (u, v) & \mapsto & (u_1, \ldots, u_k, k+1, v_1 + k, \ldots, v_l+k) & \text{ where } u = (u_1, \ldots, u_k), v = (v_1, \ldots, v_l)\\
\end{array}
\end{equation*}
Moreover, for every Catalan pair $(C,f)$, we can construct another Catalan pair $(C,f')$ with $f' : (x, y) \mapsto f(y,x)$, as it is done with $(Y)$ and $(Y')$.
\end{defn}

Then each pair of Catalan pairs gives rise to a bijection between the underlying sets. We can can construct most of the known maps like this.

\begin{ex}
\begin{itemize}
	\item $(Y) \to (\PT_1)$ gives Knuth's correspondence (or rotation map).
	\item $(Y) \to (\NCP_1)$ gives $\varphi$.
	\item $(Y) \to (Y')$ gives the mirroring map.
	\item $(Y) \to (\NCP_3)$ gives the bijection from Edelman in \cite{edelman_bij}.
	\item $(\PT_1) \to (\NCP_4)$ and $(\PT_2)\to (\NCP_1)$ give Prodinger's bijection \cite{prodinger}. 
	\item $(\PT_2) \to (\NCP_2)$ and $(\PT_1) \to (\NCP_3)$ give the bijection from Dershowitz and Zaks \cite{dershowitz_zaks}.
	\item $(\NCP_2) \to (\NCP_1)$, $(\NCP_3) \to (\NCP_4)$ and $(\NCP_5) \to (\NCP_6)$ give the Kreweras complement.
	\item $(\PT_1) \to (\RST_1)$ and $(\PT_2) \to (\RST_2)$ are "add a rightmost child to every internal vertex".
	\item $(\PT_2) \to (\LST_1)$ and $(\PT_1) \to (\LST_2)$ are "add a leftmost child to every internal vertex".
	\item $(\RST_1) \to (\NCP_1)$, $(\RST_2)\to (\NCP_6)$, $(\LST_1)\to (\NCP_3)$ and $(\LST_2)\to (\NCP_2)$ are the "gaps" maps considered by Arizmendi and Celestino in \cite{arizmendi_celestino_schroder}, and they can be extended to a map $\ST \to \NCP$.
	\item $(\NCP_7) \to (\PT_1)$ and $(\NCP_8) \to (\PT_2)$ give the map from Bernardi in \cite{bernardi}.
\end{itemize}
\end{ex}

\begin{coroll}
The following diagrams commute (all maps are bijections):
$$
\begin{tikzcd}
\PT \arrow[r, "\rot"] \arrow[d, "R"] & Y \arrow[d, "\varphi"] \\
\RST \arrow[r, "\gaps"] & \NCP
\end{tikzcd}
\hspace{30pt}
\begin{tikzcd}
\PT \arrow[r, "\rot"] \arrow[d, "L"] & Y \arrow[d, "\Edel"] \\
\LST \arrow[r, "\gaps"] & \NCP
\end{tikzcd}
\hspace{30pt}
\begin{tikzcd}
\PT \arrow[r, "\rot"] \arrow[d, "\Prod"] & Y \arrow[d, "\Edel"] \\
\NCP \arrow[r, <-, "K"] & \NCP
\end{tikzcd}$$
\end{coroll}


\section{An operadic point of view}
\label{app:operad_pt_view}

In this appendix we give an alternative definition of the maps $f_\tau$ via an operad morphism from $\overline Y$ to the set $\End_{T(B)}$, that we define below. It follows the same spirit as the operad morphism $\NCP \to \End_B$ in \cite{T-transf}, and in a way generalizes it, since the gap-insertion operad structure on $\NCP$ \cite{gap-insertion-operad} can be seen as a suboperad of $\overline Y$.

We note that an operadic viewpoint on noncrossing partition is not new, and appears for example in the amalgamation of free probability in \cite{guionnet}.

\begin{defn}
Consider the (non unitary) tensor algebra $T(B) := \oplus_{n\ge 1} B^{\otimes n}$. Elements of $T(B)^{\otimes k}$ will be written $x_1\otimes x_2 \otimes \cdots \otimes x_{n_1} | x_{n_1+1} \otimes \cdots \otimes x_{n_2}| \cdots | x_{n_{k-1}+1} \otimes \cdots \otimes x_{n_k}$. Consider the endomorphism operad $\End_{T(B)}$, where the $k$-th level are $\bb K$-linear maps $T(B)^{\otimes k} \to T(B)$, and composition is the natural one: Given $\mathcal F:T(B)^{\otimes k} \to T(B)$, $\mathcal G_1, \ldots \mathcal G_k \in \End_{T(B)}$ with $\mathcal G_i:T(B)^{\otimes j_i} \to T(B)$, the composition $\mathcal F\circ(\mathcal G_1, \ldots, \mathcal G_k) : T(B)^{\otimes (j_1 + \cdots + j_k)} \to T(B)$ is given by 
$$
	\mathcal F\circ(\mathcal G_1, \ldots, \mathcal G_k) (x_1| \ldots| x_{j_1 + \cdots + j_k}) 
	:= \mathcal F\Big( \mathcal G_1(x_1| \ldots |x_{j_1})~| \cdots|~ \mathcal G_k(x_{j_1 + \cdots + j_{k-1}+1}| \ldots |x_{j_1 + \cdots + j_k}) \Big)
$$

$T(B)$ is given a left-$B$-module structure by multiplication on the left on the first factor: $b(x_1\otimes \cdots \otimes x_n) = (bx_1)\otimes x_2 \otimes \cdots \otimes x_n$. Given a morphism of $B$-modules $f:T(B)\to B$, we can consider the operad morphism $\Phi_f : \overline Y \to \End_{T(B)}$ defined by 
\begin{align*}
	\Phi_f(\tA) &= \Big(u|v \mapsto f(u)v \Big) \\
	\Phi_f(\tB) &= \Big(u|v \mapsto u\otimes v \Big)
\end{align*}
It can be shown that the map $\Phi_f$ is well defined. Indeed, the operad $\overline Y$ is isomorphic to the duplicial operad (see \cite{frabetti}, page 390). Let us recall that the latter has a presentation with two generators, corresponding to $\tA$ and $\tB$ in $\overline Y$, and three relations, namely
\begin{align*}
	\tA \circ(\tA, \tO) &= \tA \circ(\tO, \tA), \\ 
	\tB \circ(\tB, \tO) &= \tB \circ(\tO, \tB), \\ 
	\tB \circ(\tA, \tO) &= \tA \circ(\tO, \tB).
\end{align*}
Moreover, one can compute the three relations in the operad $\End_{T(B)}$:
\begin{align*}
	\Phi_f(\tA) \circ (\Phi_f(\tA), \Phi_f(\tO)) &= \Phi_f(\tA) \circ (\Phi_f(\tO), \Phi_f(\tA)) &&: \Big(u|v|w \mapsto f(u)f(v)w \Big), \\ 
	\Phi_f(\tB) \circ (\Phi_f(\tB), \Phi_f(\tO)) &= \Phi_f(\tB) \circ (\Phi_f(\tO), \Phi_f(\tB)) &&: \Big(u|v|w \mapsto u\otimes v\otimes w \Big), \\ 
	\Phi_f(\tB) \circ (\Phi_f(\tA), \Phi_f(\tO)) &= \Phi_f(\tA) \circ (\Phi_f(\tO), \Phi_f(\tB)) &&: \Big(u|v|w \mapsto f(u)v\otimes w \Big),
\end{align*}
which therefore guarantee that $\Phi_f$ is well defined.

Note that if $f\in I\cdot \Mult[[B]]$, then $f$ can be seen as a morphism of $B$-modules $T(B)\to B$, with $f(x_1\otimes \cdots \otimes x_n) = f_n(x_1, \ldots, x_n)$. Indeed, it is left-$B$-linear since 
\begin{align*}
	f(b(x_1\otimes \cdots \otimes x_n)) 
	&= f((bx_1)\otimes\cdots\otimes x_n) \\
	&= f_n(bx_1, \ldots, x_n) = bx_1 f_n(1, \ldots, x_n) \\
	&= b f_n(x_1, \ldots, x_n) = bf(x_1\otimes\cdots\otimes x_n)
\end{align*}
for $b, x_1, \ldots, x_n \in B$.
\end{defn}

\begin{lem}
\label{lem:f_tau_operad_definition}
For $f\in I\cdot \mathrm{Mult}[[B]]$, $\tau\in Y_n$, and the word $u = x_1\otimes \cdots\otimes x_n \in T(B)$, we have
$$
	f_\tau(x_1, \ldots, x_n) = f(\Phi_f(\tau)(x_1|\cdots | x_n)).
$$
\end{lem}

\begin{proof}
It is easy to see by induction that 
$$
	\Phi_f\Big(\trightcomb n\Big)(x_1|\cdots | x_n) = x_1\otimes\cdots \otimes x_n.
$$
Moreover, for $\tau\in Y$, $\tover\tau = \tA \circ (\tau, \tO)$ so
\begin{equation}
\label{eq:phi_f_tover_simplification}
	\Phi_f(\tover \tau)(x_1|\cdots | x_n) 
	= \Phi_f(\tA)\Big(\Phi_f(\tau)(x_1|\cdots | x_{n-1}) ~\Big|~ \Phi_f(\tO)(x_n) \Big) 
	= f(\Phi_f(\tau)(x_1|\cdots | x_{n-1}))x_n.
\end{equation}
Then, for $\tau = \toverrightcomb{\tau_1}{\tau_2}{\tau_k} = \trightcomb k \circ (\tover{\tau_1}, \ldots, \tover{\tau_k})$, 
\begin{align*}
	\Phi_f(\tau)(x_1|\cdots | x_n) 
	&= \Phi_f\big(\trightcomb k\big)\Big(\Phi_f(\tover{\tau_1})(x_1|\cdots |x_{j_1}) ~\Big|~ 
	\cdots ~\Big|~ \Phi_f(\tover{\tau_k})(x_{j_{k-1}+1}| \cdots |x_n)\Big)\\ 
	&= \Phi_f(\tover{\tau_1})(x_1|\cdots |x_{j_1}) \cdots \Phi_f(\tover{\tau_k})(x_{j_{k-1}+1}|\cdots |x_n) \\ 
	&= f(\Phi_f(\tau_1)(x_1|\cdots | x_{j_1-1}))x_{j_1} \otimes \cdots \otimes f(\Phi_f(\tau_k)(x_{j_{k-1}+1}|\cdots |x_{n-1})x_n. 
\end{align*}

So, if inductively $f(\Phi_f(\tau_i)(x_{j_{i-1}+1}|\cdots | x_{j_i-1})) = f_{\tau_i}(x_{j_{i-1}+1}, \ldots, x_{j_i-1})$, then 
\begin{align*}
	f\Big(\Phi_f(\tau)(x_1|\cdots | x_n)\Big) 
	&= f\Big( f_{\tau_1}(x_1, \ldots, x_{j_1-1})x_{j_1} \otimes \cdots \otimes f_{\tau_k}(x_{j_{k-1}+1}, \ldots, x_{n-1})x_n \Big) \\ 
	&= f_k\Big( f_{\tau_1}(x_1, \ldots, x_{j_1-1})x_{j_1} , \ldots, f_{\tau_k}(x_{j_{k-1}+1}, \ldots, x_{n-1})x_n \Big) \\ 
	&= f_\tau(x_1, \ldots, x_n).
\end{align*}
\end{proof}

We can prove again Lemma \ref{lem:compo_multilinear_series_trees}, restated here:

\vspace{7pt}
\noindent\textbf{Lemma \ref{lem:compo_multilinear_series_trees}.} \textit{Let $\tau \in Y_k$, $\sigma_1 / \tO = \tover{\sigma_1}, \ldots, \sigma_k / \tO = \tover{\sigma_k}\in Y/\tO$. Let $n = \sum_{i=1}^k |\sigma_i / \tO|$. Then for $f\in I\cdot\mathrm{Mult}[[B]]$ and for all $x_1, \ldots, x_n\in B$, 
$$
	f_{\tau\circ\big(\smtover{\sigma_1},~ \ldots, \smtover{\sigma_k}\big)}(x_1, \ldots, x_n)
	= f_\tau\Big(f_{\sigma_1}(x_1, \ldots, x_{|\sigma_1|})x_{|\sigma_1|+1}, \ldots, 
	~ f_{\sigma_k}(x_{n-|\sigma_k|}, \ldots, x_{n-1})x_n\Big).
$$}
\vspace{7pt}

\begin{proof}
Let us use the notation of omitted indices. Then
\begin{align*}
	&~~f_{\tau\circ\big(\smtover{\sigma_1},~ \ldots, \smtover{\sigma_k}\big)}(x_1, \ldots, x_n) \\
	&= f_{\tau\circ\big(\smtover{\sigma_1},~ \ldots, \smtover{\sigma_k}\big)}(x, \ldots, x) \\
	&= f(\Phi_f(\tau\circ\big(\smtover{\sigma_1},~ \ldots, \smtover{\sigma_k}\big))(x | \cdots | x)) && \text{by Lemma }\ref{lem:f_tau_operad_definition}\\ 
	&= f\Big(\Phi_f(\tau)\circ \big(\Phi_f(\tover{\sigma_1}),~ \ldots, \Phi_f(\tover{\sigma_k})\big))(x | \cdots | x)\Big) && \text{as }\Phi_f \text{ is an operad morphism}\\ 
	&= f\Big(\Phi_f(\tau)\Big(\Phi_f(\tover{\sigma_1})(x| \cdots | x)~| \cdots|~ \Phi_f(\tover{\sigma_k})(x| \cdots | x)\Big)\Big) \\ 
	&= f\Big(\Phi_f(\tau)\Big(f(\Phi_f(\sigma_1)(x| \cdots | x))x~| \cdots|~ f(\Phi_f(\sigma_k)(x| \cdots | x))x\Big)\Big) && \text{by equality }\eqref{eq:phi_f_tover_simplification}\\
	&=f_\tau\Big(f_{\sigma_1}(x, \ldots, x)x, \ldots, 
	~ f_{\sigma_k}(x, \ldots, x)x\Big) && \text{by Lemma }\ref{lem:f_tau_operad_definition}.
\end{align*}
\end{proof}

\section{Conflict of interest and data availability statements}

On behalf of all authors, the corresponding author states that there is no conflict of interest.

\medskip

Data availability statement: does not apply.


\end{document}